%% file: colt_our_paper.tex
\documentclass[final,12pt]{colt2026} 

\title[DP synthetic data for smooth queries]{Minimax optimal differentially private synthetic data for smooth queries}
\usepackage{times}
\usepackage{amsmath,amssymb,mathtools}
\usepackage{algorithm,algpseudocode}
\mathtoolsset{showonlyrefs}
\usepackage{xcolor}
\usepackage{hyperref}
\usepackage{cleveref}
\input{macros.tex}
\RestyleAlgo{ruled}

\usepackage{cancel}
\usepackage[normalem]{ulem}

\coltauthor{%
 \Name{Rundong Ding} \Email{rundongd@usc.edu}\\
 \addr Department of Mathematics, University of Southern California
 \AND
 \Name{Yiyun He} \Email{yih130@ucsd.edu}\\
 \addr Department of Mathematics, University of California San Diego
 \AND
 \Name{Yizhe Zhu} \Email{yizhezhu@usc.edu}\\
 \addr Department of Mathematics, University of Southern California
}

\begin{document}

\maketitle

\begin{abstract}
Differentially private synthetic data enables the sharing and analysis of sensitive datasets while providing rigorous privacy guarantees for individual contributors. A central challenge is to achieve strong utility guarantees for meaningful downstream analysis. Many existing methods ensure uniform accuracy over broad query classes, such as all Lipschitz functions, but this level of generality often leads to suboptimal rates for statistics of practical interest. Since many common data analysis queries exhibit smoothness beyond what worst-case Lipschitz bounds capture, we ask whether exploiting this additional structure can yield improved utility.

We study the problem of generating $(\varepsilon,\delta)$-differentially private synthetic data from a dataset of size $n$ supported on the hypercube
$[-1,1]^d$, with utility guarantees uniformly for all smooth queries having bounded derivatives up to order $k$.
We propose a polynomial-time algorithm that achieves a minimax error rate of
$O_{k,d}(n^{-\min \{1, \frac{k}{d}\}})$, up to a $\log(n)$ factor.
This characterization uncovers a phase transition at $k=d$. Our results generalize the Chebyshev moment matching framework of
\citep{musco2025sharper,wang2016differentially} and strictly improve the error rates for $k$-smooth queries established in \citep{wang2016differentially}. Moreover, we establish the first minimax lower bound for the utility of $(\varepsilon,\delta)$-differentially private synthetic data with respect to $k$-smooth queries, extending the Wasserstein lower bound for $\varepsilon$-differential privacy in \citep{boedihardjo2024private}. \footnotetext[1]{Accepted for presentation at the Conference on Learning Theory (COLT) 2026.}

\end{abstract}


\section{Introduction}

Differential privacy (DP) has emerged as a leading standard for safeguarding privacy in settings that require statistical analysis over large and sensitive datasets. Introduced to provide a rigorous, attack-agnostic notion of privacy protection, differential privacy limits what an adversary can infer about any single individual, even when the adversary has substantial auxiliary information \citep{dwork2014algorithmic}. Formally, a randomized mechanism is differentially private if its output distributions are nearly indistinguishable on any pair of adjacent datasets that differ in only one individual’s record. This framework is increasingly adopted in practice, most prominently in the disclosure avoidance system for the 2020 U.S. Census \citep{abowd2019census, hawes2020implementing, hauer2021differential}, and in large-scale deployments by technology companies  \citep{cormode2018privacy, dwork2019differential}. Beyond its foundational appeal, differential privacy has enabled a broad range of data science workflows, including private query answering and workload optimization \citep{mckenna2018optimizing}, regression and empirical risk minimization \citep{chaudhuri2008privacy, su2016differentially}, parameter estimation \citep{duchi2018minimax}, and private stochastic gradient methods for modern machine learning \citep{song2013stochastic, abadi2016deep}.

Much of the existing DP literature focuses on designing task-specific mechanisms tailored to a predefined workload (e.g., a fixed set of queries or a particular model class). While often highly accurate, these approaches can require substantial expertise, ranging from sensitivity analysis and mechanism selection to careful privacy accounting across iterative computations. A complementary and increasingly practical alternative is differentially private synthetic data generation, where one releases a synthetic dataset that approximates statistical properties of the original data while satisfying differential privacy \citep{hardt2012simple, bellovin2019privacy, wasserman2010statistical, barak2007privacy}. By the post-processing property \citep{dwork2014algorithmic} of differential privacy, analysts can subsequently perform a wide range of downstream tasks on the released synthetic dataset without incurring additional privacy loss.

\subsection{Differentially private synthetic data}

Differentially private synthetic data consist of artificially generated records that aim to
preserve the aggregate statistical structure of a sensitive source dataset (often user-generated),
while providing rigorous privacy guarantees to the individuals who contributed to the data
\citep{ponomareva2025dp}. By decoupling data utility from direct access to raw records, DP synthetic data offers a principled way to unlock the value of datasets that are otherwise difficult to share due to privacy concerns.

The task of generating private synthetic data can be formulated as follows.
Let $(\Omega,\rho)$ be a metric space and consider a dataset
$ X=(X_1,\dots, X_n)\in \Omega^n$.
The goal is to design an efficient randomized algorithm that outputs a differentially private
synthetic dataset $Y=(Y_1,\dots,Y_m)\in \Omega^m$ such that the empirical measures
\[
  p_{X}=\frac{1}{n} \sum_{i=1}^n \delta_{X_i}
  \quad \text{and} \quad
  p_{Y}=\frac{1}{m} \sum_{j=1}^m \delta_{Y_j}
\]
are close under an appropriate notion of discrepancy.
A large body of prior work evaluates the utility of DP synthetic data by requiring accuracy for a
prescribed finite collection of queries
\citep{hardt2010multiplicative,hardt2012simple,bun2017make,bun2023continual,gonzalez2024mirror}.
At the same time, there are fundamental computational barriers: under standard cryptographic
assumptions, \citep{ullman2020pcps} showed that generating differentially private
synthetic Boolean data that are simultaneously accurate for rich classes of Boolean queries  is computationally
intractable in general.

Motivated by the desire for task-agnostic utility guarantees, a more recent line of work measures
the quality of synthetic data in a continuous space $\Omega$ via
$\mathbb{E}\, W_1(p_{X},p_{Y})$, where $W_1$ denotes the $1$-Wasserstein distance and the
expectation is taken over the randomness of the algorithm.
By the Kantorovich--Rubinstein duality (see, e.g., \citep{villani2009optimal}),
\begin{equation}
\label{eq:KR_duality}
  W_1(p_X,p_Y)
  = \sup_{\mathrm{Lip}(f)\le 1}
  \left(\int f\,\dd p_X-\int f\,\dd p_Y\right),
\end{equation}
where the supremum ranges over all $1$-Lipschitz functions on $\Omega$.
Since many learning objectives and evaluation functionals are Lipschitz, or can be controlled by
Lipschitz losses
\citep{von2004distance,kovalev2022lipschitz,bubeck2021universal,meunier2022dynamical},
\eqref{eq:KR_duality} yields a uniform performance guarantee for a broad range of downstream tasks
performed on synthetic datasets whose empirical distribution is close to $p_X$ in $W_1$.

Within this Wasserstein framework, \citep{boedihardjo2024private} gave the first $\varepsilon$-DP
algorithm achieving the minimax rate $n^{-1/d}$ (up to polylogarithmic factors) for data
in the hypercube $[0,1]^d$ under the $\ell_{\infty}$ metric.
Subsequently, \citep{he2023algorithmically} developed two classes of algorithms that attain the same
minimax-optimal $n^{-1/d}$ rate for $d\ge 2$.
These constructions have enabled the analysis of several downstream learning tasks
\citep{gu2025differentially,wirth2025private,cao2025differentially}. Under the relaxed notion of $(\varepsilon,\delta)$-DP, \citep{musco2025sharper} analyzed a Chebyshev
moment matching algorithm that achieves the same $n^{-1/d}$ Wasserstein rate.

Recent work has sought to surpass the worst-case $n^{-1/d}$ scaling by exploiting additional
structure in the data.
For example, \citep{he2025differentially} considered datasets concentrated near a $d'$-dimensional
subspace and obtained a rate $n^{-1/d'}$ up to polynomial factors in $d$, while
\citep{donhauser2024certified} studied data supported on an unknown $d'$-dimensional manifold in
$[0,1]^d$.
When the dataset satisfies suitable sparsity conditions, \citep{holland2025private} showed that the
error rate can be further improved via an adaptation of the Private Measure Mechanism introduced in
\citep{he2023algorithmically}.
In a related direction, \citep{gonzalez2025private} analyzed the practical Private Evolution algorithm
\citep{lin2023differentially} through its connection to the Private Signed Measure Mechanism in \citep{he2023algorithmically}. Relatedly, \citep{feldman2024instance} studied instance-optimal private density estimation under
Wasserstein loss, and \citep{he2024online} extended these ideas to online synthetic data release.

Beyond Wasserstein distance, several alternative discrepancy measures have been considered,
including $L_1$-type distances \citep{wasserman2010statistical} and more general integral probability
metrics (IPMs) \citep{asadi2023gibbs}.
The $1$-Wasserstein distance is a particular example of an IPM, and
\citep{bousquet2020synthetic} characterized the sample complexity of DP synthetic data generation for
unbounded query classes under IPMs.
Moreover, \citep{donhauser2024certified} analyzed the class of $s$-sparse queries and obtained an
$n^{-1/s}$ error rate.

While Wasserstein-based guarantees are appealing due to their uniform control over all
$1$-Lipschitz test functions, they can be overly pessimistic for smoother statistics.
In practice, many summaries of interest correspond to smooth queries with bounded derivatives,
such as low-order moments and correlations $f(\bx)=\prod_{j\in J\subset [d]} x_j$,  linear statistics
$f(\bx)=\langle v,\bx\rangle$ \citep{blum2013learning,hardt2010multiplicative},
quadratic functionals  $f(x)=\bx^\top A \bx$ arising in covariance and
correlation estimation \citep{amin2019differentially},
and smooth loss or score functions used in private regression, classification, and convex optimization
\citep{chaudhuri2011differentially,ullman2015private} such as Gaussian kernel function $f(\bx)=\exp(\frac{-\|\bx-\mu\|^2}{2\sigma^2})$ and logistic function $f(x)=\frac{1}{1+e^{-\langle v,\bx\rangle }}$.
For such queries, utility guarantees tailored to smoothness can be substantially sharper than
worst-case Lipschitz bounds \citep{wang2016differentially}. This observation motivates the following question:
\begin{center}
  \textit{Can we generate DP synthetic data with utility guarantees for smooth queries that improve
    the  $n^{-1/d}$ rate?}
\end{center}

Prior work \citep{wang2016differentially} studied the release of $k$-smooth queries under differential
privacy and achieved an error rate of
$O\left(n^{-\frac{2k}{3d+2k}}\right)$ (up to its dependence on the privacy parameter) under $(\varepsilon,\delta)$-DP and a rate $\left(n^{-\frac{k}{2d+k}}\right)$ under $\varepsilon $-DP.
In the special case $k=1$, this rate does not recover the minimax-optimal $n^{-1/d}$ exponent
established in \citep{boedihardjo2024private}.
In contrast, we improve the dependence on the sample size to
$\tilde O_{k,d}\!\left(n^{-\min \{1, \frac{k}{d}\}}\right)$ and show that this rate is minimax optimal.  
Our results thus uncover a refined trade-off between query smoothness and achievable utility
guarantees for differentially private synthetic data.

\subsection{Main results}
We will consider $(\varepsilon,\delta)$-differential privacy defined below.
\begin{definition}[$(\varepsilon,\delta)$-Differential Privacy]
A randomized algorithm $\mathcal{M} : \Omega^n \to \mathcal{Z}$ is said to satisfy
$(\varepsilon,\delta)$-differential privacy if for all datasets
$ X,  X' \in \Omega^n$ that differ in at most one entry
and for all measurable sets $S \subseteq \mathcal{Z}$,
\[
\mathbb{P}\!\left( \mathcal{M}( X) \in S \right)
\;\le\;
e^{\varepsilon}\,
\mathbb{P}\!\left( \mathcal{M}( X') \in S \right)
\;+\;
\delta .
\]
\end{definition}
We focus on the class of smooth queries. For an integer $k \ge 1$, define the \textit{$k$-smooth function class}

    \begin{equation}\label{eq:def_Fk}
\mathcal{F}_k \coloneqq
\Bigl\{\, f : [-1,1]^d \to \mathbb{R} \;\Bigm|\;
\max_{0\le |\balpha| \le k}\, \lVert \partial^{\balpha} f \rVert_\infty \le 1
\,\Bigr\}
\end{equation}
and the associated \textit{integral probability metric} \citep{sriperumbudur2009integral} between two probability distribution $p,q$ on $[-1,1]^d$:
\begin{equation}
\label{eq: def_dk}
    d_k(p, q) \coloneqq \sup_{f \in \mathcal{F}_k} \big| \langle f, p - q \rangle \big| = \sup_{f \in \mathcal{F}_k} \left|\int f \dd p -\int f\dd q \right|.
\end{equation}
When \(k=1\), the metric \(d_1\) is equivalent to the \(1\)-Wasserstein distance
on \([-1,1]^d\) equipped with the \(\ell_2\)-metric. Indeed, by the
Kantorovich--Rubinstein duality \cite{villani2009optimal} and a standard smooth approximation argument,
\begin{align}\label{eq:Wd_equivalence}
    W_1(p,q)\le d_1(p,q)\le \sqrt d\, W_1(p,q).
\end{align}
The upper inequality follows because every \(f\in\mathcal F_1\) is
\(\sqrt d\)-Lipschitz with respect to \(\ell_2\), while the lower inequality
follows since smooth \(1\)-Lipschitz functions are contained in \(\mathcal F_1\)
and are dense in the class of all \(1\)-Lipschitz functions.

 For any data set $X=(x_1,\ldots,x_n)\in \Omega^n$, where $\Omega=[-1,1]^d$ equipped with the $\ell_2$-metric, we denote its associated empirical distribution as
  \[
    p_X \coloneqq  \frac1n \sum_{j=1}^n \delta_{x_j}.
  \]
The utility of a given differentially private synthetic data generation algorithm  $\mathcal M$ is given by $d_{k}(p_{X}, p_{\mathcal M(X)})$, where $\mathcal M(X)\subseteq \Omega$ is the output of the algorithm $\mathcal M$.

Now we are ready to state our algorithmic results:  

\begin{theorem}[DP synthetic data generation for $k$-smooth queries]
\label{thm:main}
Assume  $\varepsilon \in (n^{-1},1)$, $\delta\in (0,1)$ and $
\frac{\varepsilon n}{\sqrt{\log(1.25/\delta)}}\geq (ck)^{\max\{d,k\}}$ for some absolute constant $c>1$.
    There exists a polynomial-time $(\varepsilon,\delta)$-DP algorithm that outputs a synthetic dataset $Y\subseteq \Omega$ such that 
    \begin{align}
        \E d_k(p_{X},p_Y)\leq 
        \begin{cases}
             C_{d,k}   \left(\dfrac{\sqrt{\log(1.25/\delta)}}{\varepsilon n}{}\right)^{k/d} &\text{if }k<d,\\
            \dfrac{C_{d,k}\sqrt{\log(1.25/\delta)}\log (\varepsilon n)}{\varepsilon n} &\text{if }k=d,\\ 
            \dfrac{C_{d,k}\sqrt{\log(1.25/\delta)}}{\varepsilon n} &\text{if }k>d.
        \end{cases} 
    \end{align}
   Here  $C_{d,k} \le (C k)^k d^{\max\{1,k/2\}}$ for a universal constant $C>0$.
\end{theorem}

The full differentially private synthetic data generation procedure is presented in Algorithm~\ref{alg: main}. In particular, the size of the synthetic dataset $Y$ is  $O(n)$ when $k<d/2$ and is $O(n^2)$  in all cases (see Remark~\ref{rmk:choiceofm} for detailed discussion).  

 To obtain a nontrivial utility bound in Theorem~\ref{thm:main}, we require $k,d=o(\log n)$.
 The computational complexity is
    $O\left((\varepsilon n)^{(k+2)\min\{k,d\}}\right)$, which is polynomial in $n$ for fixed $k,d$.
We refer the reader to Appendix~\ref{sec:complexity} for further discussion. This complexity is consistent with the hardness result of \citep{ullman2020pcps}: under standard cryptographic assumptions, they show that there is no polynomial-time algorithm for generating differentially private synthetic data accurate for all two-way marginal when $n$ is polynomial in $d$. 

  In terms of the privacy parameter $\varepsilon$, for simplicity, we assume \(\varepsilon<1\) in Theorem~\ref{thm:main}. When \(\varepsilon\geq 1\), the only modification is the choice of the variance in the Gaussian mechanism step of Algorithm~\ref{alg: main}. Based on Lemma~\ref{lem: gaussian_mechanism}, the same analysis applies, differing only in its dependence on $\varepsilon$ in the utility bound.
  
\begin{remark}[Comparison with \citep{he2023algorithmically}] When $k=1$, Theorem~\ref{thm:main} recovers the convergence rate established in \citep{he2023algorithmically} for the $1$-Wasserstein distance due to \eqref{eq:Wd_equivalence}. The only difference is an additional factor of $d$, which arises from the underlying choice of metric: \citep{he2023algorithmically} considers $[-1,1]^d$ equipped with the $\ell_\infty$ metric, whereas our analysis is carried out on $[-1,1]^d$ endowed with the $\ell_2$ metric.
\end{remark}

Theorem~\ref{thm:main} uncovers a sharp phase transition in the role of smoothness: once the smoothness order exceeds the dimension ($k>d$), additional smoothness no longer yields further improvements in the error rate. In particular, for multivariate polynomial queries, including low-order moments and correlations, linear statistics, and quadratic functionals, the theorem establishes an error rate of $O_{k,d}(n^{-1})$ by choosing the smoothness parameter $k>d$. In this regime, the sample-size exponent becomes one; the associated
constants may still depend on $k$ and $d$.
This strictly improves the $O_{k,d}\left(n^{-\frac{2k}{3d+2k}}\right)$ rate in \citep[Theorem 22]{wang2016differentially} when $k=1$, and recovers \citep[Theorem 41]{musco2025sharper}.

The next theorem shows that the rate achieved in Theorem \ref{thm:main} is minimax optimal up to a factor in $k,d$.

\begin{theorem}[Minimax lower bound]\label{thm:main2}  For any constant $c_1>1$, assume $\varepsilon \in (\frac{c_1}{n},c_1), \delta\in [0, c_2(e^\varepsilon-1))$ for some constant $c_2$ depending only on $c_1$.  Then the following holds:
    \begin{equation}
\inf_{\mathcal{M}\text{ is } (\varepsilon,\delta)\text{-DP}}\; \sup_{X\in \Omega^n} \E \big[d_k(p_X,p_{\mathcal{M}(X)})\big]
\geq C_{k} (n\varepsilon)^{-\min\{1,k/d\}},
\end{equation}
where the infimum is over all $(\varepsilon,\delta)$-algorithms that output synthetic data $\mathcal M(X)\subseteq \Omega$, and $C_{k}$ is a constant depending only on $k$ and $c_1$.
\end{theorem}

Our minimax lower bound is more general than the lower bound in
\citep{boedihardjo2024private}, which studies $\varepsilon$-DP algorithms with
utility measured by the $W_1$ distance under the $\ell_\infty$ metric on
$\Omega$. Since functions in $\mathcal F_1$ are $d$-Lipschitz with respect to the
$\ell_\infty$ metric, specializing our result to $\delta=0$ and $k=1$ recovers
the $n^{-1/d}$ Wasserstein lower bound of \citep{boedihardjo2024private}, up to a
factor of $d$. Since obtaining a non-trivial bound
$n^{-1/d}=O(1)$ requires $d=O(\log n)$, our result recovers \citep{boedihardjo2024private} up to a $\log n$ factor.
Finally, our proof strategy is fundamentally different from the packing-number
argument used in \citep[Corollary~9.4]{boedihardjo2024private}.

Theorem~\ref{thm:main2} covers a wide range of $(\varepsilon,\delta)$ used in practice. In particular, $\varepsilon$ can be a constant larger than 1 and $\delta$ can be of order 1. This is in contrast to many existing minimax lower bounds in the literature with restricted ranges of $(\varepsilon,\delta)$.
The restriction on $\delta$ in minimax lower bounds is typically not inherent to the definition of
$(\varepsilon,\delta)$-DP; rather, it reflects (i) the privacy regime one intends to model and
(ii) how the proof composes privacy losses across many individuals, steps, or hypotheses.
Broadly, prior techniques fall into two  regimes:

\begin{itemize}
    \item Vanishing $\delta$.
    In fingerprinting-code and tracing-based lower bounds for private query answering
    (e.g., \citep{bun2018fingerprintingcodespriceapproximate,steinke2015pureapproximatedifferentialprivacy}),
    the argument repeatedly applies approximate DP along a long adaptive process.
    As a result, the $\delta$-failure probability can accumulate (typically via union bounds) across many steps,
    and one therefore requires $\delta \ll 1/n$ to keep the overall failure probability negligible.
    A related requirement appears in the estimation lower-bound framework of \citep{cai2020costprivacyoptimalrates},
    where a tracing attack combined with a group-privacy/union-bound step leads to the assumption
    $\delta=o(1/n)$.

    \item {$\varepsilon$-relative $\delta$.}
    In DP-Le Cam and DP-Assouad \citep{acharya2021differentially}, approximate DP enters primarily through an additive slack term that
    scales with the Hamming radius of the reduction (via group privacy), so it suffices to control $\delta$
    \emph{relative} to $\varepsilon$.
    For example, \citep{acharya2021differentially} yields a distinguishability contribution of the form $D\delta$;
    when the construction takes $D=\Theta(1/\varepsilon)$, assuming $\delta=O(\varepsilon)$ keeps this term bounded.
    Likewise, \citep[Theorem~1.3]{kamath2023newlowerboundsprivate}  assumes $\delta\le \varepsilon$
    so that the accumulated approximate-DP slack remains $O(1)$.
\end{itemize}
    Our lower bound falls into this second regime: we only need privacy leakage to remain uniformly controlled
    throughout the reduction, and the condition $\delta \lesssim (e^{\varepsilon}-1)$ in Theorem~\ref{thm:main2}
    ensures that the accumulated slack does not
    dominate.

\paragraph{Technical overview}

In the work \citep{musco2025sharper}, the Chebyshev moment matching method captures information from a dataset $X$ via the Chebyshev moments of $p_X$. By perturbing these moments, one can construct a private probability measure $q$. The utility loss $W_1(p_X,q)$ is controlled through the analytic properties of Chebyshev polynomial approximations for 1-Lipschitz functions; in particular, Jackson's theorem~\citep{jackson1930theory} is a key ingredient in achieving the $n^{-1/d}$ rate in~\citep{musco2025sharper}. This connection between moment matching and Jackson's theorem was also recently explored by \citep{amini2026wasserstein} in the context of empirical measure concentration under the Wasserstein distance.

In the proof of Theorem~\ref{thm:main}, we leverage the $k$-smoothness of the function class $\mathcal F_k$ to achieve Chebyshev polynomial approximation rates superior to those provided by Jackson's theorem for Lipschitz functions. In our setting, a generalized Jackson's theorem for $k$-smooth functions is needed, and we provide an improved estimate based on \citep{jackson1912approximation} tailored to our setting in  Lemma~\ref{lem:Jackson}.  In addition, we derive a higher-order global coefficient-decay result for multivariate Chebyshev expansions. Our primary technical contribution is an argument that translates these smoothness assumptions into global energy decay bounds (Theorem \ref{thm:main-d}). While this result is central to our utility analysis of DP-synthetic data in Theorem~\ref{thm:main}, it may also be of independent interest for broader distribution approximation problems. Finally, in contrast to \citep{musco2025sharper}, which outputs a private probability measure, our algorithm produces an explicit differentially private synthetic dataset Y. This representation is more directly usable in downstream applications, and we explicitly quantify the size of Y in the proof of Theorem~\ref{thm:main}.

  Our analysis for $(\varepsilon,\delta)$-DP relies on the $\ell_2$-sensitivity analysis based on the Gaussian mechanism, which does not directly apply to the $\ell_1$-sensitivity analysis based on the Laplacian mechanism \citep{dwork2014algorithmic}. Achieving an optimal rate for $\varepsilon$-DP synthetic data generation for $k$-smooth queries remains an open question.

In Theorem~\ref{thm:main2}, we establish a minimax lower bound for synthetic data mechanisms by tailoring a DP-Assouad (hypercube) reduction \citep{acharya2021differentially} to the $d_k$-distance, and the main novelties are the extra ingredients needed to make that reduction both privacy-aware and smoothness-aware. First, we develop an explicit way to turn approximate-DP indistinguishability into a total-variation distance control for outputs when two input datasets differ in many entries. Second, we construct a family of localized smooth bump queries supported on disjoint grid cells, and a matching family of hard datasets. Finally, we connect privacy to utility by a coupling-based Assouad inequality \citep{yu1997assouad}. Optimizing the grid resolution then yields the lower-bound rate and explains the phase transition where additional smoothness stops improving the achievable accuracy.

 \paragraph*{Organization of the paper}
 The rest of the paper is organized as follows. We provide preliminary results on differential privacy and Chebyshev polynomials in Section~\ref{sec: preliminaries}. In Section~\ref{sec:IPM}, we establish a general theorem connecting the integral probability metric induced by $\mathcal F_k$ and Chebyshev moments between two probability measures.
 In Section~\ref{sec:alg}, we present our DP-synthetic data generation algorithm and analyze its privacy and utility guarantees stated in Theorem~\ref{thm:main}. 
 The proof of the minimax lower bound (Theorem~\ref{thm:main2}) is given in Appendix~\ref{sec:minimax}. Additional proofs are deferred to the Appendices~\ref{sec:isometry} to \ref{sec: convergence_rate}. The computational complexity of Algorithm~\ref{alg: main} is analyzed in Appendix~\ref{sec:complexity}.

\section{Preliminaries}
\label{sec: preliminaries}

\subsection{Differential privacy}

The following lemma provides a simple yet useful algorithm to guarantee $(\varepsilon, \delta)$-differential privacy. It suffices to focus on the \emph{sensitivity} of our target and adding Gaussian noise to the true value.

\begin{lemma}[Gaussian Mechanism \citep{dwork2014algorithmic, balle2018improving,vinterbo2021closedformscalebound}]
\label{lem: gaussian_mechanism}
    Let $f:\Omega^n \to \R^d$ with $\ell^2$-sensitivity defined as 
    \[\Delta_{2,f}^2 \coloneqq \sup_{X,X'\text{ adjacent}} \|f(X)-f(X')\|_2^2.\]
    Then the randomized algorithm $\mathcal M: X\mapsto f(X)+Z$ with $Z\sim N(0,\sigma^2 I_d)$ satisfies $(\varepsilon,\delta)$-differential privacy if
    \[\sigma = \frac{\sqrt2\Delta_{2,f}}{\varepsilon}\cdot \sqrt{\log\frac1{4\delta(1-\delta)} + \varepsilon} \quad \text{for }\delta < 1/2.\]
    Moreover, if we assume both $\varepsilon, \delta\in (0,1)$, we can take 
    $\sigma = \frac{\Delta_{2,f} \sqrt{2\log(1.25/\delta)}}{\varepsilon}$.
\end{lemma}

\subsection{Normalized multivariate Chebyshev polynomials}
We will use multivariate Chebyshev polynomials in our algorithm and its analysis. Let us start with the 1-dimensional case. On $[-1,1]$ let
\[
  d\mu(x)=\frac{dx}{\pi\sqrt{1-x^2}},
\]
and on $[-1,1]^d$ the product probability measure $\mu_d=\mu^{\otimes d}$.
Write $\btheta=(\theta_1,\dots,\theta_d)\in[0,\pi]^d$ and $\cos\btheta=(\cos\theta_1,\dots,\cos\theta_d)$.
For $n\ge0$ define the \textit{orthonormal  Chebyshev basis}
\[
  \overline T_0(\bx)\equiv1; \qquad \overline T_n(\bx)=\sqrt2\,\cos(n\arccos x),\ \forall n\ge1.
\]
For a multi-index $K=(k_1,\dots,k_d)\in\mathbb{N}^d$ define the normalized Chebyshev polynomial
\[
  \overline T_K(\bx)=\prod_{i=1}^d \overline T_{k_i}(x_i).
\]
Let $\psi_0(\theta)\equiv1$ and $\psi_n(\theta)=\sqrt2\cos(n\theta)$ for $n\ge1$, and set
\[
  \psi_K(\btheta)=\prod_{i=1}^d \psi_{k_i}(\theta_i).
\]
All $L^2$ norms and inner products of the functions in $\btheta$-space are taken with respect to the probability measure $\pi^{-d}d\btheta$ on $[0,\pi]^d$:
\[
\ip{u}{v}:=\frac1{\pi^d}\int_{[0,\pi]^d} u(\btheta)\,v(\btheta)\,\dd\btheta,\quad
\norm{u}_2^2:=\ip{u}{u}, \quad \forall u=u(\btheta), v=v(\btheta).
\]

The following lemma reveals the relation between the two $L^2$ spaces we defined. It indicates that the normalized Chebyshev polynomials $\{\overline{T}_K\}$ serve as a role of Fourier basis on $[-1,1]$ equipped with weight measure $\mu_d$. The proof of Lemma~\ref{lem:isometry} is in Appendix~\ref{sec:isometry}.

\begin{lemma}[Isometry to cosine variables and orthonormal basis]\label{lem:isometry}
$\{\psi_K\}_{K\in\mathbb{N}^d}$ is an orthonormal basis of $L^2([0,\pi]^d)$. The linear map $U:L^2(\mu_d)\to L^2([0,\pi]^d, \pi^{-d}\dd\btheta)$, $(Uf)(\btheta)\coloneqq f(\cos\btheta)$, is an isometry:
\[
  \|Uf\|_2^2=\|f\|_{L^2(\mu_d)}^2.
\]
Moreover, $U(\overline T_K)=\psi_K$ and hence $\{\overline{T}_K\}_{K\in\mathbb{N}^d}$ is an orthonormal basis of $L^2(\mu_d)$. 
\end{lemma}

\section{Integral probability metric and Chebyshev moments}\label{sec:IPM}

For a multi-index $\balpha=(\alpha_1,\dots,\alpha_d)$, denote
$|\balpha|\coloneqq\sum_i\alpha_i$. For $K=(k_1,\dots,k_d)$, denote
 $\|K\|_2^2\coloneqq \sum_i k_i^2$. We provide an estimate of the coefficients in the Chebyshev expansion of a $k$-smooth function in $[-1,1]^d$. The proof of Theorem~\ref{thm:main-d} is in Appendix~\ref{sec:appendix_chebyshev}.

\begin{theorem}[Order-$k$ coefficient decay]
\label{thm:main-d}
    Let $d\ge1$, $k\in\mathbb{N}$, and $f\in C^k([-1,1]^d)\subset L^2(\mu_d)$. With the expansion
    \[
      f=\sum_{K\in\mathbb{N}^d} c_K\,\overline T_K,\qquad
      c_K=\ip{f}{\overline T_K}_{L^2(\mu_d)},
    \]
    there is 
    \begin{equation}\label{eq:final} 
      \sum_{K}\|K\|_2^{2k}\,c_K^2
      \;\le\; C_{d,k}'\,
      \Big(\max_{0\leq |\bbeta|\le k}\|\partial_x^\bbeta f\|_{L^\infty([-1,1]^d)}\Big)^2,
    \end{equation}
    where $C_{d, k}' 
    \le d^k e^{2k}(k!)^2$.
\end{theorem}

We also need the following generalization of Jackson's Theorem for $k$-smooth functions in $[-1,1]^d$. This can be derived from \citep{jackson1912approximation,carothers1998short}, and we include the proof in Appendix~\ref{sec:Jackson}.
\begin{lemma}[Jackson's Theorem for $d$-dimensional $k$-smooth functions]\label{lem:Jackson}
Let $d,k\geq 1$ and $f\in \mathcal F_k$ with expansion $f(\bx)=\sum_K c_K\overline T_K(\bx)$. Then for every integer  $m\geq k$, there exists a tensor-product Chebyshev polynomial 
\[\tilde f_m(\bx)=\sum_{K\in\{0,1,\dots,m\}^d} \tilde c_K \overline T_K(\bx),\] such that 
\[ \|f-\tilde f_m\|_\infty \ \le\ \frac{d C_k^{\mathrm{Jac}}}{m^k}; \qquad|\tilde c_K|\leq |c_K|, \;\forall K,\]
where $C_k^{\mathrm{Jac}}\leq (Ck)^k$ is a constant depending only on $k$. 
\end{lemma} 

Next, we apply the result of the Chebyshev moments of $k$-smooth functions and study the relationship between two types of discrepancies of probability measures $p$ and $q$: the integral probability metric $d_k(p,q)$ and the weighted $\ell_2$-difference of their Chebyshev moments. The proof of Theorem~\ref{thm: IPM_Chebyshev} is deferred to Appendix~\ref{sec:IPM_proof}.
\begin{theorem}[$d_k$-distance from Chebyshev moments]
    \label{thm: IPM_Chebyshev}
    For two probability measures $p,q$, let 
    \[\Gamma^2 \coloneqq \sum_{K\in \{0,...,m\}^d\setminus\{0\}} \frac{1}{\|K\|_2^{2k}} \left\vert \E_{X\sim p} \overline{T}_K(X) - \E_{X\sim q} \overline{T}_K(X)\right\vert ^2.\]
    Then  for any integer $k\geq 1$ and $m\geq k$,
    \[d_k(p, q) \coloneqq \sup_{f \in \mathcal{F}_k} \big| \langle f, p - q \rangle \big| \leq \frac{2C_k^{\mathrm{Jac}}\cdot d}{m^k} + \sqrt{C_{d,k}'}\cdot\Gamma,\]
    where $C_{d,k}'$ is the constant defined in Theorem~\ref{thm:main-d}. 
\end{theorem}

As a corollary of  Theorem~\ref{thm: IPM_Chebyshev}, we can consider the case where $p=q_n$ is the empirical measure of $n$ i.i.d samples of a population density $q$. Then the following theorem shows 
the convergence rate of the empirical measures to $q$ under the $d_k$-metric. The proof is deferred to Appendix~\ref{sec: convergence_rate}.
\begin{theorem}
    \label{thm: convergence_rate}
    For a probability measure $q$ supported on $[-1,1]^d$, let $\bx_1,\dots, \bx_n$ be i.i.d samples from $q$ and let $q_n$ denote the empirical distribution of $\{\bx_i\}_{i=1}^n$. Then for the metric $d_k$ define in \eqref{eq: def_dk},
    \[\E d_k(q_n,q)\leq \begin{cases}
        C_{d,k}^{\mathrm{conv}} n^{-k/d}, & \text{if } 2k<d,\\
        C_{d,k}^{\mathrm{conv}}\sqrt{\log n/ n}, & \text{if }2k=d,\\
        C_{d,k}^{\mathrm{conv}}/\sqrt{n}, &\text{if }2k>d,
    \end{cases}\]
    where $C_{d,k}^{\mathrm{conv}}\leq (Ck\sqrt{d})^k$ and $C$ is an absolute constant.
\end{theorem}

\section{DP synthetic data generation for $k$-smooth queries}\label{sec:alg}

\begin{algorithm}[ht]
\DontPrintSemicolon
\caption{Private Chebyshev moment matching for smooth queries}
\KwIn{Dataset $\bx_1,\dots,\bx_n \in [-1,1]^d$, privacy parameters $\varepsilon,\delta>0$, smoothness parameter $k\geq 1$. Moment matching level $m>0$.}

1: Let $\Delta = 2 m^{-k}$. Partition $[-1,1]^d$ into a uniform grid of side length $\Delta$ along each coordinate. Define $\mathcal{G}$ to be the set of cell centers, i.e.,
\[
\mathcal{G}
=
\left\{
-1 + \left(i_1 + \tfrac{1}{2}\right)\Delta,\,
\dots,\,
-1 + \left(i_d + \tfrac{1}{2}\right)\Delta
:\;
i_j = 0,1,\dots,\tfrac{2}{\Delta}-1
\right\}.\]

Set $r=m^k$ be the number of points per coordinate,
and for $J=(j_1,\dots,j_d)\in [r]^d$ denote 
\[\bg_J=\Big(-1+\frac{2j_1-1}{2}\cdot\Delta,\dots, -1+\frac{2j_d-1}{2}\cdot\Delta\Big )\]
to be the $J$-th element of $\mathcal{G}$.\;

2: For $i=1,\dots,n$, round $\bx_i$ to the nearest point in $\mathcal G$: \[\tilde{\bx}_i=\arg\min_{\by\in \mathcal{G}} \|\bx_i-\by\|_2.\]

3: Set $\sigma^2=   8\cdot 2^d \cdot\frac{S\log(1.25/\delta)}{n^2 \varepsilon^2}$, where $S=\sum_{K\in \{0,\dots,m\}^d  \setminus \{\0\}  }\frac{1}{\|K\|_2^k}$.\\

\quad  For $K\in \{0,\dots,m \}^d\setminus \{\0\}$, let $\widehat{m}_K=\eta_K+\tfrac{1}{n}\sum_{i=1}^n \overline T_K(\tilde{\bx}_i)$ with $\eta_K \sim  N(0, \|K\|_2^k \sigma^2)$.

4: Let $\{q_J\}_{J\in [r]^d}$ be the approximate solution to the following optimization problem:
\begin{align*}
\min_{(z_J)_ {J \in [r]^d}} \quad &
\sum_{K \in \{0,\ldots,m\}^d \setminus \{\0\}}
\frac{1}{\|K\|_2^{2k}}
\left(
\widehat{m}_K
- \sum_{J \in [r]^d} z_J \, \overline{T}_K(\bg_J)
\right)^{\!2}, \\[4pt]
\text{subject to} \quad &
\sum_{J \in [r]^d} z_J = 1,
\qquad
z_J \ge 0, \ \forall J \in [r]^d
\end{align*}
with error $\eta \leq (\varepsilon n)^{-2}$.
Define probability measure $q:=\sum_{J\in [r]^d}  q_J \delta_{g_J}$.

5: Let $Y$ be the multiset of $m'$ i.i.d. samples from the probability measure $q$ with
    \[m'\coloneqq \begin{cases}
        \lceil \varepsilon n \rceil & \text{if }2k<d,\\
        \lceil (\varepsilon n)^2 \rceil, &\text{if }2k\geq d.
    \end{cases}\]

\KwOut{ 
The multiset $Y$ of size $m'$.} 
\label{alg: main}
\end{algorithm}

We now present Algorithm~\ref{alg: main}, which generalizes the synthetic data algorithm in \citep{musco2025sharper} to $k$-smooth queries. Algorithm~\ref{alg: main} is motivated by Theorem~\ref{thm: IPM_Chebyshev}: by perturbing the Chebyshev moments of the original empirical data distribution, we can obtain a synthetic data distribution with good utility guarantees.  The outline of the algorithm can be summarized as follows: 
\begin{enumerate}
    \item Partition the region $[-1,1]^d$ into a grid with resolution $\Delta$.
    \item Align the original data to the grid points and get an approximate distribution on the grid. 
    \item Compute the Chebyshev moments and add noise to guarantee privacy.
    \item Generate a synthetic data distribution $q$ by matching the perturbed Chebyshev moments.
    \item Output a synthetic data set $Y$ where $p_Y$ is an empirical version of the distribution $q$.
\end{enumerate}

Next, we show the privacy and accuracy guarantee of Algorithm~\ref{alg: main}, which proves Theorem~\ref{thm:main}.


\subsection{Privacy}
For a dataset  $X = \{\bx_1, \ldots, \bx_n\} \subseteq [-1,1]^{d}$, let $f(X)$ be a vector-valued function with the scaled Chebyshev moments
of the uniform distribution over $X$, indexed by $K \in \{0, \ldots, m\}^d \setminus \{\0\}$:
\[
f(X)_K = \frac{1}{\|K\|_2^{k/2}} \cdot \frac{1}{n} \sum_{i=1}^n \overline{T}_K(\bx_i),
\]
where $f(X)_K$ denotes the $K$-th entry of the vector $f(X)_K$, and $\overline{T}_K(\bx)$ is the $K$-th normalized
multivariate Chebyshev polynomial.  
Define $\nnz(K)$ to be the number of non-zero entries in $K$. Since $\max_{\bx \in [-1,1]^d} |\overline{T}_K(\bx)| \le  {2^{\mathrm{nnz}(K)/2}}$, for adjacent datasets $X,X'$, we have
\[|f({X})_K-f({X}')_K|\leq \frac{1}{\|K\|_2^{k/2}} \cdot \frac{2}{n}\cdot \max_{\bx \in [-1,1]^d} |\overline{T}_K(\bx)|\leq \frac{ {2\cdot 2^{\mathrm{nnz}(K)/2}}}{\|K\|_2^{k/2}\cdot n}.\]
Hence, the sensitivity has bound
\begin{align*}
    \Delta_{2,f}^2
    & = \max_{\substack{{X}, {X}' \text{ adjacent}}}
    \| f({X}) - f({X}') \|_2^2 \\
    & \le
    \sum_{K \in \{0, \ldots, m\}^d \setminus \{\0\}}
    \frac{1}{\|K\|_2^k} \cdot \frac{1}{n^2} \cdot
     {4 \cdot 2^{\mathrm{nnz}(K)}} \\
    & \le
     {\frac{4 \cdot 2^d}{ n^2}}
    \sum_{K \in \{0, \ldots, m\}^d \setminus \{\0\}} \frac{1}{\|K\|_2^k}=
     {\frac{4 \cdot 2^d}{ n^2}} S,
\end{align*}
where $S = \sum_{K \in \{0, \ldots, m\}^d \setminus \{\0\}} \frac{1}{\|K\|_2^k}$ as defined in Algorithm~\ref{alg: main}.
Let the $g(X)$ denote the rounded datasets computed Algorithm~\ref{alg: main}, i.e. $g(X) \coloneqq \{\tilde{\mathbf{x}}_1, \ldots, \tilde{\mathbf{x}}_n\}$. Then for two adjacent datasets ${X}, {X}'$, we have $g(X), g(X')$ are also adjacent. Thus we have 
\[\Delta_{2,f\circ g}^2 \leq \Delta_{2,f}^2 \leq  {\frac{4 \cdot 2^d}{ n^2}} S.\]
By Gaussian mechanism (Lemma~\ref{lem: gaussian_mechanism}) and our choice of $\sigma^2$ in Algorithm~\ref{alg: main},   adding independent $N(0,\sigma^2)$ to each component of $f\circ g(x)$ guarantees $(\varepsilon, \delta)$-differential privacy, which is exactly to output $(\|K\|_2^{-k/2} \widehat m_K)_{K \in \{0, \ldots, m\}^d \setminus \{\0\}}$. Then the moments $(\widehat m_K)_{K \in \{0, \ldots, m\}^d \setminus \{\0\}}$ satisfies $(\varepsilon, \delta)$-differential privacy. 
Finally, since the remainder of Algorithm~\ref{alg: main} simply post-processes
$(\widehat{m}_K)$
without returning to the original data $X$,
the output of the algorithm is also $(\varepsilon, \delta)$-differentially private.

\subsection{Accuracy}
\label{sec: Accuracy}
Next, we show the accuracy bound under the integral probability metric $d_k$ in the following steps.

\noindent\textbf{(1) Error in Step 2:}
 Let $\widetilde{{X}}$ denote the aggregated dataset in Step 2. Since each data point is moved by at most $\frac{\sqrt d}{2}\Delta$ in $\ell_2$-distance,  and $f\in \mathcal F_k$ is $\sqrt{d}$-Lipschitz under $\ell_2$-distance,
\[d_k(p_{{X}}, p_{\widetilde{{X}}}) =\sup_{f\in \mathcal F_k}  \big| \langle f, p_{{X}} - p_{\widetilde{{X}}}\rangle \big|\leq \sup_{i} \sup_{f\in\mathcal F_k}|f(\bx_i)-f(\overline{\bx}_i)| \leq \sqrt{d} \sup_i \|\bx_i-\overline{\bx}_i\| \leq \frac{d\Delta}{2}
.\]

\noindent\textbf{(2) Error in Step 4:}
Since $q$ is approximately the optimal construction under the weighted $\ell^2$-norm of the Chebyshev moments,  
\begin{align*} 
    \Gamma^2 & \coloneqq \sum_{K} \frac{\left\vert \E_{X\sim p_{\widetilde {{X}}}} \overline{T}_K(X) - \E_{X\sim q} \overline{T}_K(X)\right\vert ^2}{\|K\|_2^{2k}} \\
    &\leq 2\sum_{K} \frac{\left\vert \E_{X\sim p_{\widetilde{{X}}}} \overline{T}_K(X) - \widehat{m}_K\right\vert^2 + \left\vert \widehat{m}_K - \E_{X\sim q} \overline{T}_K(X)\right\vert ^2}{\|K\|_2^{2k}} \\
    &\leq 4 \sum_{K} \frac{\left\vert \E_{X\sim p_{\widetilde{{X}}}} \overline{T}_K(X) - \widehat{m}_K\right\vert ^2}{\|K\|_2^{2k}} + 2\eta\\
    & = 4 \sum_{K} \frac{\left\vert \eta_K\right\vert ^2}{\|K\|_2^{2k}} + 2\eta.
\end{align*}
And by Theorem~\ref{thm: IPM_Chebyshev}, when $m\geq k$, with  the constant $C_{d,k}'$   in Theorem~\ref{thm:main-d}, there is 
\[d_k(p_{\widetilde{{X}}}, q)\leq \frac{2 C_k^{\mathrm{Jac}}\cdot d}{m^{k}} + \sqrt{C_{d,k}'}\cdot \Gamma.\]
For the term with $\Gamma$, taking expectations, 
\[(\E \Gamma)^2 \leq \E \Gamma^2 
= 4 \sum_{K} \frac{\E \left\vert \eta_K\right\vert ^2}{\|K\|_2^{2k}}  + 2\eta 
= 4 \sum_{K} \frac{\sigma^2}{\|K\|_2^{k}}  + 2\eta
= 32\cdot 2^d \cdot\frac{\log(1.25/\delta)}{n^2 \varepsilon^2} \cdot S^2 + 2\eta.\]
Applying the estimate of $S$ from Lemma~\ref{lem: est_S} in Appendix~\ref{sec:S}, we have  
\begin{align}\E \Gamma &
\leq  
c^d\cdot \frac{\sqrt{\log(1.25/\delta)}}{\varepsilon n} \cdot 
    \begin{cases}
    1+\dfrac{m^{\,d-k}-1}{d-k}, & \text{if } k<d,\\
    1+\log m, & \text{if }k=d, \phantom{\dfrac{1}{1}}\\
    1+\dfrac{1}{k-d}, &\text{if }k>d,
    \end{cases}
\end{align}
where $c>\sqrt{8}$ is an absolute constant.

\noindent\textbf{(3) Error in Step 5:}
Since the synthetic data in $Y$ are $m'$ i.i.d samples from distribution $q$, by Theorem~\ref{thm: convergence_rate} we have 
\[\E d_k(p_Y, q) \leq \begin{cases}
            C_{d,k}^{\mathrm{conv}} m'^{-k/d}, & \text{if } 2k<d,\\
        C_{d,k}^{\mathrm{conv}}\sqrt{\log m'/ m'}, & \text{if }2k=d,\\
        C_{d,k}^{\mathrm{conv}}/\sqrt{m'}, &\text{if }2k>d.
\end{cases}\]
Taking
\begin{equation}
    m'=\begin{cases}
        \lceil \varepsilon n \rceil & \text{if }2k<d,\\
        \lceil (\varepsilon n)^2 \rceil, &\text{if }2k\geq d
\end{cases}
\label{eq: m'choice}
\end{equation}
suffices to obtain $\E d_k(p_Y, q) \leq C_{d,k}^{\mathrm{conv}}(\varepsilon n)^{-\min\{1, k/d\}}$.

\noindent\textbf{(4) Total error:}
Together, we have an accuracy bound 
\begin{align*}
    \E d_k(p_{{X}}, p_Y) & \;\leq\; \E d_k(p_{{X}}, p_{\widetilde{{X}}}) + \E d_k(p_{\widetilde{{X}}}, q)  + \E d_k(q,p_Y)\\
    &\leq\; d  \Delta + \frac{2C_{k}^{\mathrm{Jac}}\cdot d}{m^{k}} + \sqrt {C_{d,k}'}\cdot \E \Gamma+C_{d,k}^{\mathrm{conv}}(\varepsilon n)^{-\min\{1,k/d\}}.
\end{align*}
In the three different cases, by choosing $m= \Big \lceil \frac{1}{c} \cdot \bigg(
\frac{\varepsilon n}{\sqrt{\log(1.25/\delta)}}\bigg)^{1/\max\{d,k\}} \Big\rceil \geq k$ due to our assumption $
\frac{\varepsilon n}{\sqrt{\log(1.25/\delta)}}\geq (ck)^{\max\{d,k\}}$,  we have 
\begin{enumerate}
    \item When $k<d$, $(c\cdot m)^d \leq \frac{\varepsilon n}{\sqrt{\log(1.25/\delta)}}$, and there is
    \begin{align*}
        \E d_k(p_{{X}}, p_Y) & \leq   d   \Delta + \frac{2C_k^{\mathrm{Jac}}\cdot d}{m^{k}} +  \sqrt {C_{d,k}'}\cdot \frac{c^d\sqrt{\log(1.25/\delta)}}{\varepsilon n} \cdot m^{d-k}+ C_{d,k}^{\mathrm{conv}}(\varepsilon n)^{-k/d} \\
        &\leq  C_{d,k} \cdot \bigg(\frac{\sqrt{\log(1.25/\delta)}}{\varepsilon n}\bigg)^{k/d};
    \end{align*}    
    \item When $k=d$, there is
    \begin{align*}
        \E d_k(p_{{X}}, p_Y) & \leq  d  \Delta + \frac{2C_k^{\mathrm{Jac}}\cdot d}{m^{k}} +  \sqrt {C_{d,k}'}\cdot \frac{c^d\sqrt{\log(1.25/\delta)}}{\varepsilon n} \cdot \log m +C_{d,k}^{\mathrm{conv}}(\varepsilon n)^{-1}\\
        &\leq C_{d,k} \cdot \frac{\sqrt{\log(1.25/\delta)}}{\varepsilon n}\cdot\log (\varepsilon n);
    \end{align*}
    \item When $k>d$, there is
    \begin{align*}
        \E d_k(p_{{X}}, p_Y) & \leq  d  \Delta + \frac{2C_k^{\mathrm{Jac}}\cdot d}{m^{k}} +  \sqrt {C_{d,k}'}\cdot \frac{c^d\sqrt{\log(1.25/\delta)}}{\varepsilon n} \cdot 2 +C_{d,k}^{\mathrm{conv}}(\varepsilon n)^{-1} \\
        &\leq  C_{d,k} \cdot \frac{\sqrt{\log(1.25/\delta)}}{\varepsilon n}. 
    \end{align*}
\end{enumerate}

Here in each of the three cases, we have a uniform bound for the constant $C_{d,k}$:
\[C_{d,k} \leq c^{k} \left(2d + 2C_k^{\mathrm{Jac}}\cdot d + \sqrt{C_{d,k}'} + C_{d,k}^{\mathrm{conv}} \right) \leq (Ck)^k d^{\max\{1, k/2\}}\]
by using the constant bounds in Theorem~\ref{thm:main-d}, Theorem~\ref{thm: IPM_Chebyshev}, and Theorem~\ref{thm: convergence_rate}.

\acks{The authors thank Roman Vershynin for helpful comments. Y.Z. was partially supported by the Simons Grant MPS-TSM-00013944.}

\bibliography{ref}


\appendix

\section{Minimax lower bound}\label{sec:minimax}

In this section, we prove Theorem~\ref{thm:main2} in the following steps.

\subsection{Total variation distance bound}
We will first introduce several auxiliary lemmas that connect the definition of differential privacy with total variation distance.

\begin{lemma}[TV bound from approximate likelihood ratio]\label{lem:tv-tanh}
Let $P$ and $Q$ be probability measures on a common measurable space
$(\Omega,\mathcal F)$. Suppose that there exist parameters $\xi\ge 0$ and
$\tau\ge 0$ such that for every measurable set $S\in\mathcal F$,
\begin{equation}\label{eq:approx-lr-assumption}
P(S)\ \le\ e^{\xi} Q(S) + \tau,
\qquad
Q(S)\ \le\ e^{\xi} P(S) + \tau.
\end{equation}
Then
\begin{equation}\label{eq:tv-tanh-simple}
\operatorname{TV}(P,Q)
\coloneqq \sup_{S\in\mathcal F} |P(S)-Q(S)|
\ \le\ \tanh\!\Big(\frac{\xi}{2}\Big)+\frac{2\tau}{e^{\xi}+1}
.
\end{equation}
\end{lemma}

\begin{proof}
We first note that
\begin{equation}\label{eq:tv-positive-part}
\operatorname{TV}(P,Q)
=\sup_{S\in\mathcal F} |P(S)-Q(S)|
=\sup_{S\in\mathcal F} \big(P(S)-Q(S)\big),
\end{equation}
because for any measurable $S$ we have
\[
P(S^c)-Q(S^c)
=1-P(S)-\big(1-Q(S)\big)
= -\big(P(S)-Q(S)\big).
\]
Thus, it suffices to bound $P(S)-Q(S)$ from above for any measurable set $S$. For simplicity, fix $S$ and denote
\[
a\coloneqq P(S),\qquad b\coloneqq Q(S), 
\]
Now applying \eqref{eq:approx-lr-assumption} to $S$ and $S^c$ separately, we have that 
\[
P(S)-Q(S)-\tau \ \le\ (e^{\xi}-1)b,
\qquad
P(S)-Q(S)-\tau \ \le\ (e^{\xi}-1)(1-a).
\] 
Combining the above two inequalities, 
\[
P(S)-Q(S)
\ \le\ \frac{e^{\xi}-1}{e^{\xi}+1} + \frac{2\tau}{e^{\xi}+1}
= \tanh\!\Big(\frac{\xi}{2}\Big)\ +\ \frac{2\tau}{e^{\xi}+1}.
\]
Taking the supremum over $S$ finishes the proof.
\end{proof}

To control the total variation distance between two output datasets under $(\epsilon, \delta)$-DP, we iterate the DP guarantee along a Hamming path  between two datasets as follows:
\begin{lemma}[TV distance between two DP-synthetic data outputs] \label{lem:dp-tv-fixed-distance}
Let $\mathcal M$ be an $(\varepsilon,\delta)$-differentially private mechanism that takes an input dataset in $\Omega^n$ and outputs a synthetic dataset.
Let $X,Y\in\Omega^n$ be two fixed datasets with Hamming distance
\[
h\ \coloneqq \ d_{\mathrm{Ham}}(X,Y)=\#\{j : x_j \neq y_j\},
\]
and denote $Q_X, Q_Y$ as the probability measures of the distributions of the outputs $\mathcal{M}(X), \mathcal{M}(Y)$ correspondingly. Then 
\begin{equation}\label{eq:vi-rho}
\operatorname{TV}(Q_X,Q_Y) \le \tanh\!\Big(\frac{\varepsilon h}{2}\Big)\Big(1+\frac{2\delta}{e^{\varepsilon}-1}\Big).
\end{equation}
\end{lemma}

\begin{proof}
Since $d_{\mathrm{Ham}}(x,y)=h$, there exists a path
\[
X=X^{(0)},X^{(1)},\dots,X^{(h)}=Y
\]
in the Hamming graph on $\Omega^n$ such that each consecutive pair differs in exactly one coordinate.
Write $Q^{(t)}\coloneqq \mathcal M\big(X^{(t)}\big)$ for $t=0,\dots,h$.
By $(\varepsilon,\delta)$--DP, for every measurable $S$ and every $t$,
\[
Q^{(t+1)}(S)\ \le\ e^\varepsilon Q^{(t)}(S)+\delta.
\]
Iterating this inequality for $t=0,1,\dots,h-1$ and unrolling the recursion yields
\[
Q_Y(S)\ =\ Q^{(h)}(S)
\ \le\ e^{\varepsilon h}Q^{(0)}(S)+\delta\sum_{j=0}^{h-1} e^{j\varepsilon}
\ =\ e^{\varepsilon h}Q_X(S)+\delta_h,
\]
where $\delta_h=\delta\sum_{j=0}^{h-1} e^{j\varepsilon}
\ $. Reversing the path from $y$ back to $x$ gives
\begin{align}
    Q_X(S) \leq e^{\varepsilon h}Q_Y(S)+\delta_h,
\end{align}
where   
\[
\delta_h=\delta\sum_{j=0}^{h-1} e^{j\varepsilon}
=\delta\frac{e^{\varepsilon h}-1}{e^\varepsilon-1}.
\]
 Then Lemma~\ref{lem:tv-tanh} yields
\[
\operatorname{TV}(Q_X,Q_Y)
\ \le\ \tanh\!\Big(\frac{\varepsilon h}{2}\Big) + 2\delta\left(\frac{e^{\varepsilon h}-1}{e^\varepsilon-1}
\right)\cdot \frac{1}{e^{\varepsilon h}+1}= \tanh\!\Big(\frac{\varepsilon h}{2}\Big) \Big(1+\frac{2\delta}{e^{\varepsilon}-1}\Big).\]
 \end{proof}

\subsection{Construction of a family of test functions}

In this section, we construct a collection of test functions in $\mathcal F_k$ defined in \eqref{eq:def_Fk}, which will later be used in the lower bound proof.

\begin{lemma}[A class of test functions in $\mathcal F_k$]\label{lem:local-bumps-en}
Fix integers $d\ge 1$ and $k\ge 1$, and let $m\ge 2$.
Partition $[-1,1]^d$ into $M=m^d$ many uniform closed axis-aligned cubes $\{O_t\}_{t=1}^M$ of side length $r=\frac{2}{m}$, and let $\ba_t$ be the center of $O_t$. 
 For $1\leq t\leq M$, there exists $f_t\in C_c^\infty(O_t)\cap \mathcal{F}_k$  and some point $\bx_t^{\max} \in O_t$ such that
  \[ f_t(\ba_t)=0,   \quad f_t(\bx_t^{\max}) \geq c_{k}^{\max} r^k,\]
  where $c_{k}^{\max}$ is a constant depending only on $k$.
\end{lemma}

\begin{proof}
    Taking a smooth function $\eta\in C_c^\infty(\R)$ such that 
    \[
\eta(u)\equiv 1 \ \text{for } u\in\Big[-\frac18,\frac18\Big],
\qquad
\operatorname{supp}(\eta)\subset \Big[-\frac14,\frac14\Big].
\]
Define a $d$-dimensional bump function
\[
\chi(\bu) := u_1 \prod_{\ell=1}^d \eta(u_\ell), \qquad \bu=(u_1,\dots,u_d)\in\R^d,
\]
and set the normalization constant 
\[
C_0 := \max_{1\leq |\balpha|\le k}\big\|\partial^\balpha \chi\big\|_\infty \in (0,\infty).
\]
From our construction of $\chi (\bu)$, $C_0$ is a constant depending only on $k$. (It suffices to consider the case $k<d$. Then we are taking derivatives to at most $k$ many $\eta(u_\ell)$, and we can upper bound the rest by $|\eta(u_{\ell'})|\leq 1$ for $\ell'\neq \ell$.) 

For each $t\in[M]$, define
\[
f_t(\bx) := \frac{r^k}{C_0} \chi \left(\frac{\bx-\ba_t}{r}\right).
\]

\textbf{(i)Support and disjointness. } Since $\operatorname{supp}(\chi)\subset [-\tfrac14,\tfrac14]^d$, we have 
\[
\operatorname{supp}(f_t)\subset \ba_t + r\Big[-\frac14,\frac14\Big]^d \subset O_t.
\]
Because these inner boxes lie strictly inside the grid cells, the supports are pairwise disjoint.

\textbf{(ii) bounds on derivatives} For any multi-index $\balpha$  with $1\leq |\balpha|\le k$, by the chain rule, 
\[
\partial^\balpha f_t(\bx)
= \frac{r^{k-|\balpha|}}{C_0}\,
(\partial^\balpha\chi)\!\left(\frac{\bx-\ba_t}{r}\right).
\]
Hence
\[
\|\partial^\balpha f_t\|_\infty
\le \frac{r^{k-|\balpha|}}{C_0}\,\|\partial^\balpha \chi\|_\infty
\le \frac{1}{C_0}\max_{1 \le |\bbeta|\le k}\|\partial^\bbeta\chi\|_\infty
=1,
\]
where we used $r \le 1$ and $k-|\balpha| \ge 0$. Therefore $f_t\in \mathcal F_k$. 

\textbf{(iii) A lower bound at one point.} Since $\chi(0)=0$, we have $f_t(\ba_t) = 0$. Define
\[
\bu^\star := \Big(\frac{1}{16},0,\dots,0\Big),\qquad \bx_t^{\max}:=\ba_t + r \bu^\star.
\]
Then, $\bu^\star \in[-\tfrac18,\tfrac18]^d $, so $\eta(u^\star_\ell)=1$ for all $\ell$ and thus 
\[
f_t(\bx_t^{\max})
= \frac{r^k}{C_0} \chi(\bu^\star)
= \frac{r^k}{16C_0}.
\]
Finally, setting $c^{\max}_{k} = 1/(16C_0) > 0$ gives $f_t(\bx_t^{\max})\ge c^{\max}_{k}r^k$. 
\end{proof}

\subsection{Construction of  input datasets}
We construct $2^M$ many different datasets $X^{\theta}, \theta\in \{\pm 1\}^M$ as follows:
\begin{construction}
\label{construction: lower_bound_construction}
    Given integers $d\ge1$, $k\ge1$ and a grid size $m\ge 2$, let $r\coloneqq 2/m$ and $M\coloneqq m^d$. Let $\{O_t\}_{t=1}^M$ be the uniform partition of $[-1,1]^d$ into closed hypercubes of length $r$ with centers $(\ba_t)$. Let $\{f_t\}_{t=1}^M$ be the function class introduced in Lemma~\ref{lem:local-bumps-en}.
  Assume $\frac{n}{M}$ is an integer.
 For any fixed $t\in [M]$, define the multi-sets 
        \begin{align*}
            X_t^{-1} &\coloneqq \{ n_t \text{ copies of }\ba_t\},\\
            X_t^{+1} &\coloneqq \{n_t - \lfloor \beta n\rfloor \text{ copies of }\ba_t\} \cup \{\lfloor \beta n\rfloor \text{ copies of } \bx_t^{\max}\},
        \end{align*}
        where $\bx_t^{\max} \in O_t$ satisfies $f_t(\bx_t^{\max})\geq c_{d,k}^{\max} r^k$, and $\sum_{t=1}^M n_t=n,  \min_{1\leq t\leq M} n_t \geq \frac{n}{2M}$.
        For every binary sequence $\theta\in\{\pm 1\}^M$, define a new multi-set 
        $X^\theta \coloneqq \bigcup_{t=1}^M X_t^{\theta_t}$
        and denote its corresponding empirical measure as $p_{X^\theta}$.
\end{construction}

The following lemma shows that with the input datasets from Construction~\ref{construction: lower_bound_construction}, if a type of function class exists, this implies a lower bound under the $d_k$-distance. 
Then, analogous to the general DP-Assouad method in \citep{acharya2021differentially}, we can insert the TV-distance bound into  Assouad's method to control the $d_k$-distance. 

\begin{lemma}[Assouad-type lower bound]\label{lem:assouad-empirical}
Consider a family of deterministic datasets
$\{X^\theta\}_{\theta\in\{\pm1\}^M}$ under Construction \ref{construction: lower_bound_construction}.
For a given $\theta \in \{\pm 1\}^M$, and $i\in\{1,\dots,M\}$, let $\theta^{(i,+)},\theta^{(i,-)}\in \{\pm 1\}^M$ be the two vectors with $\pm 1$ in the $i$-th coordinate, and  all other coordinates coincide with $\theta$.  Assume that there are functions
$\{f_i\}_{i=1}^M\subseteq \mathcal F_k$ such that:

\begin{enumerate}
\item The supports of $\{f_i\}_{i=1}^M$ are pairwise disjoint:
$\operatorname{supp}(f_i)\cap\operatorname{supp}(f_j)=\varnothing$ for $i\neq j$;
\item For all $1\leq i\leq M$, the following integral is independent of $\theta\in \{\pm 1\}^M$:
\begin{equation}\label{eq:emp-sep-def}
  \tau_i = 
  \int f_i(\bx)\,\dd\big(p_{X^{\theta^{(i,+)}}}-p_{X^{\theta^{(i,-)}}}\big)(\bx) \;>\;0.
\end{equation}
\end{enumerate} 
Sample $\theta\sim \mathrm{Unif}(\{\pm 1\}^M)$.
Let $\mathcal M$ be any algorithm that maps 
$X^\theta$ to an output dataset $Z\subset \mathcal X$. Let $Q_{(i,+)}$ and $Q_{(i,-)}$ denote the conditional distribution of $Z$ conditioning on the events $\{\theta_i=+1\}$ and $\{\theta_i=-1\}$, respectively. Define
\begin{equation}\label{eq: def_vi}
    v_i\coloneqq 
    \operatorname{TV}\big(Q_{{(i,+)}},Q_{{(i,-)}}\big).
\end{equation}
Then 
\begin{equation}\label{eq:assouad-main}
  \E_{\theta, Z}\Big[d_k\big(p_{X^\theta},p_Z \big)\Big]
  \;\ge\;
  \frac12\sum_{i=1}^M (1-v_i)\tau_i.
\end{equation}

\end{lemma}

\begin{proof}
For each $i$, define the linear functional
\[
\Delta_i(\theta,Z)
\coloneqq \int f_i\,\dd\big(p_{X^\theta}-p_Z \big).
\]

Given any sign vector $s=(s_1,\dots,s_M)\in\{\pm1\}^M$, consider
$f_s\coloneqq \sum_{i=1}^M s_i f_i$. By the disjointness of supports, for every
multi-index $\alpha$ with $|\alpha|\le k$ and every $x\in \mathcal X$ there is at most
one $i$ such that $x\in\operatorname{supp}(f_i)$, hence
\[
\partial^\alpha f_s(x)
=\sum_{i=1}^M s_i\,\partial^\alpha f_i(x)
=\partial^\alpha f_{i_0}(x)
\]
for some $i_0$, and therefore
\(
|\partial^\alpha f_s(x)|
\le \|\partial^\alpha f_{i_0}\|_\infty\le 1.
\)
Thus $f_s\in \mathcal F_k$. Then for any $s\in \{\pm 1\}^M$, the following lower bound holds:
\[
d_k\big(p_{X^\theta},p_Z \big)
=\sup_{f\in \mathcal F_k}\left|\int f\,\dd\big(p_{X^\theta}-p_Z \big)\right|
\;\ge\;
\left|\int f_s\,\dd\big(p_{X^\theta}-p_Z \big)\right|
=\Big|\sum_{i=1}^M s_i\,\Delta_i(\theta,Z)\Big|.
\]
For the given pair $(\theta,Z)$, choosing
$s_i\coloneqq \operatorname{sgn}(\Delta_i(\theta,Z))$ (with an arbitrary choice if
$\Delta_i=0$) gives
\[
\Big|\sum_{i=1}^M s_i\,\Delta_i(\theta,Z)\Big|
=\sum_{i=1}^M |\Delta_i(\theta,Z)|,
\]
and we obtain  
\begin{equation}\label{eq:dk-sum-delta}
d_k\big(p_{X^\theta},p_Z \big)
\;\ge\;\sum_{i=1}^M |\Delta_i(\theta,Z)|.
\end{equation}
Next, we will bound the right-hand side of \eqref{eq:dk-sum-delta} in expectation. 
Fix index $i$. 
For simplicity, denote 
\[
\int f_i\,\dd p_{X^{\theta^{(i,+)}}}\eqqcolon a,
\qquad
\int f_i\,\dd p_{X^{\theta^{(i,-)}}}\eqqcolon b,
\qquad a-b=\tau_i>0,
\]
and 
\[T_i(Z)\coloneqq \int f_i\,\dd p_Z .\] 
Recall the following conditional distribution we defined:
\[
Z\mid(\theta_i=+1)\sim Q_{{(i,+)}},\qquad
Z\mid(\theta_i=-1)\sim Q_{{(i,-)}},
\]
and $v_i=
    \operatorname{TV}\big(Q_{{(i,+)}},Q_{{(i,-)}}\big)$.
Let $(Z_{(i,+)},Z_{(i,-)})$ be a maximal coupling of $Q_{{(i,+)}}$ and $Q_{{(i,-)}}$ so that \[\mathbb P(Z_{(i,+)}= Z_{(i,-)})= 1-v_i.\]
On the event $\{Z_{(i,+)}=Z_{(i,-)}\}$, we have $T_i(Z_{(i,+)})=T_i(Z_{(i,-)})$ hence
\[
\frac12\left(\Big|a-T_i(Z_{(i,+)})\Big|+\Big|b-T_i(Z_{(i,-)})\Big|\right)
\;\ge\; \frac{|a-b|}{2}\,\mathbf 1_{\{Z_{(i,+)}=Z_{(i,-)}\}}
\;=\; \frac{\tau_i}{2}\,\mathbf 1_{\{Z_{(i,+)}=Z_{(i,-)}\}}.
\]
Taking expectations with respect to the coupling gives
\[
\frac12\,\E\left[\Big|a-T_i(Z_{(i,+)})\Big|+\Big|b-T_i(Z_{(i,-)})\Big|\right]
\;\ge\; \frac{\tau_i}{2}(1-v_i).
\]

Then we can take expectations in \eqref{eq:dk-sum-delta}. Since $\theta_i\sim \mathrm{Uniform} \{\pm 1\}$, by revealing the $i$-th coordinate of $\theta$ first, we further obtain
\begin{align}
\E_{\theta,Z}|\Delta_i(\theta,Z)|
&=\frac12\E_{Z\sim Q^{(i,+)}}\Big|a -T_i(Z)\Big|
+\frac12\E_{Z\sim Q^{(i,-)}}\Big|b-T_i(Z)\Big|\\
&=\frac12\E \left(\Big|a-T_i(Z_{(i,+)})\Big|+\Big|b-T_i(Z_{(i,-)})\Big|\right)
\;\ge\; \frac{\tau_i}{2}(1-v_i).
\end{align}
Summing over $i$, we obtain
\[
\E_{\theta,Z}\Big[d_k\big(p_{X^\theta},p_Z \big)\Big]
\;\ge\;
\sum_{i=1}^M \E_{\theta,Z}\big|\Delta_i(\theta,Z)\big|
\;\ge\;\frac12\sum_{i=1}^M (1-v_i)\,\tau_i,
\]
as desired.
\end{proof}

\subsection{Proof of Theorem~\ref{thm:main2}}

Putting all ingredients together, we will finish the proof of Theorem~\ref{thm:main2} based on Lemma~\ref{lem:local-bumps-en}, Lemma~\ref{lem:assouad-empirical}, and datasets from Construction~\ref{construction: lower_bound_construction}.

\begin{proof}(Proof of Theorem~\ref{thm:main2})
Let $\mathcal M$ be any 
$(\varepsilon,\delta)$--DP algorithm that outputs synthetic data in $[-1,1]^d$. Let $\theta\sim\mathrm{Unif}(\{\pm1\}^M)$, $M=m^d$, $Z\coloneqq \mathcal M(X^\theta)$, and $p_Z $ be its associated empirical distribution. 
Let $Q_{(i,+)}$ and $Q_{(i,-)}$  be the   distribution of $Z$ conditioning on the events $\{\theta_i=+1\}$ and $\{\theta_i=-1\}$, correspondingly, and let
\[v_i\coloneqq 
    \operatorname{TV}\big(Q_{{(i,+)}},Q_{{(i,-)}}\big) , \quad 1\leq i\leq M.\]

{\bf Step 1: Uniform upper bound for $v_i$.}
For any $\theta\in \{-1,1\}^M$, we associate a pair of vectors $(\theta^{(i,+)},\theta^{(i,-)})$ that differs only on the $i$-th coordinate and coincide with $\theta$ for all other coordinates.

For any constant $c_1>0$  and $\varepsilon\in (0,c_1)$, choosing $\beta=\frac{2c_1}{n\varepsilon}$ in Construction~\ref{construction: lower_bound_construction}, we have
\[d_{\mathrm{Ham}}(X^{\theta^{(i,+)}}, X^{\theta^{(i,-)}}) \leq \lfloor \beta n\rfloor \leq 2c_1/\varepsilon.\]

Let $Q_{X^{\theta^{(i,+)}}},Q_{X^{\theta^{(i,-)}}}$ be the distribution of $\mathcal{M}(X^{\theta^{(i,+)}}), \mathcal{M}(X^{\theta^{(i,-)}})$. Then by Lemma \ref{lem:dp-tv-fixed-distance}, 
under the assumption 
\[  0\leq \varepsilon\leq c_1, \quad \delta\leq c_2(e^{\varepsilon}-1), \quad \text{where} \quad  0<c_2<\frac{1}{2} \left( \frac{1}{\tanh (c_1)}-1\right),\]
 we have
\begin{align}\operatorname{TV}(Q_{X^{\theta^{(i,+)}}},Q_{X^{\theta^{(i,-)}}}) &\leq \tanh\left(\frac{\varepsilon \lfloor \beta n\rfloor}{2}\right) (1+\frac{2\delta}{e^{\varepsilon}-1})\\
&\leq \tanh(c_1) (1+\frac{2\delta}{e^{\varepsilon}-1})\\
&<\tanh(c_1) (1+2c_2)<1.
\end{align}
Denote the constant $c_3:=\tanh(c_1) (1+2c_2)<1$. Since $Q_{(i,+)}, Q_{(i,-)}$ is the weighted sum of all probability measures $Q_{X^{\theta^{(i,+)}}},Q_{X^{\theta^{(i,-)}}}$ over all but $i$-th coordinate,
\[
v_i = \operatorname{TV}\big(Q_{{(i,+)}},Q_{{(i,-)}}\big) \leq \frac{1}{2^{M-1}} \!\! \sum_{(\theta^{(i,+)},\theta^{(i,-)})} \operatorname{TV}(Q_{X^{\theta^{(i,+)}}},Q_{X^{\theta^{(i,-)}}}) 
\leq c_3 < 1,
\]
where the sum is over all pairs of $\{(\theta^{(i,+)},\theta^{(i,-)})\}, \theta\in \{\pm 1\}^M$ that differs only on the $i$-th coordinate.

\medskip

{\bf Step 2: Apply Lemma~\ref{lem:assouad-empirical}. }
We next check our function class $\{f_t\}_{t=1}^M$ introduced in Lemma~\ref{lem:local-bumps-en} satisfies the three conditions  from Lemma~\ref{lem:assouad-empirical}.

Conditions (i), (ii) hold by Lemma~\ref{lem:local-bumps-en}. For Condition (iii), note that for any pair of $(\theta^{(i,+)},\theta^{(i,-)})$ that differs only on the $i$-th coordinate, by Construction~\ref{construction: lower_bound_construction}, the differences between datasets $X^{\theta^{(i,+)}}$ and $X^{\theta^{(i,-)}}$ are exactly $\lfloor\beta n\rfloor$ many points at either $\bx_i^{\max}$ or $\ba_i$, which is independent of the values of $\theta^{(i,+)},\theta^{(i,-)}$ on the other coordinates. Hence,
\[
\tau_i=\int f_i\,\dd\big(p_{X^{\theta^{(i,+)}}}-p_{X^{\theta^{(i,-)}}}\big)
= \frac{\lfloor\beta n\rfloor}n \left(f_i(\bx_i^{\max}) - f_i(\ba_i) \right)
\]
is a constant independent of $\theta$. Since $\beta=\frac{2c_1}{n\varepsilon}$, and $0<\varepsilon\leq c_1$, from Lemma~\ref{lem:local-bumps-en},  
\begin{align}
    \tau_i\geq  \frac{\beta n-1}n \left(c_{k}^{\max}r^k - 0 \right)  \geq\frac{\beta}{2} c_{k}^{\max} r^k.
\end{align}
Recall $r=2/m$ from Construction~\ref{construction: lower_bound_construction}. This implies
 \[\tau_i \geq C_{k}\beta m^{-k}\] for some constant $C_{k}>0$ depending on $k$.
Plugging the bounds for $v_i, \tau_i$ into 
\eqref{eq:assouad-main} in Lemma~\ref{lem:assouad-empirical}, and using $M=m^d$, we obtain
\begin{align}\label{eq:bayes_lower_bound}
\E_{\theta,Z}\big[d_k(p_\theta,p_Z )\big]
\;\ge\; MC_k\beta m^{-k} \cdot (1-c_3)
\gtrsim  C_{k} \frac{m^{d-k}}{n\varepsilon}.
\end{align}

{\bf Step 3: Optimize over the grid size $m$.}
 Since there are $n_t\geq \frac{n}{2M}$ many data points in each cell $O_t$ from Construction~\ref{construction: lower_bound_construction}, the requirement that each cell contains at least $\lfloor \beta n\rfloor$ movable points is $\beta n\le n/m^d$. Since $\beta=\frac{2c_1}{n\varepsilon}$, this reads $m^d\lesssim n\varepsilon$. Thus all
admissible $m$ satisfy $1\le m\lesssim (n\varepsilon)^{1/d}$. Define \[\Phi(m)\coloneqq m^{d-k}/(n\varepsilon).\] If $d>k$ then $\Phi(m)$ is increasing
in $m$, so the optimum is attained at $m\asymp (n\varepsilon)^{1/d}$ and
\[
\Phi(m)\asymp \frac{(n\varepsilon)^{(d-k)/d}}{n\varepsilon}
=(n\varepsilon)^{-k/d}.
\]
If $d=k$, then $\Phi(m)\equiv (n\varepsilon)^{-1}$ is independent of $m$. If
$d<k$, then $\Phi(m)$ is decreasing in $m$, so the optimum is attained at the
smallest admissible $m$ (a fixed constant), which again yields
$\Phi(m)\asymp (n\varepsilon)^{-1}$. Optimizing in $m$ for \eqref{eq:bayes_lower_bound}
yields
\begin{equation}\label{eq:data-risk-rate}
\E_{\theta,Z}\big[d_k(p_\theta,p_Z )\big]
\;\gtrsim\;
\begin{cases}
C_{k} (n\varepsilon)^{-k/d}, & d>k,\\
C_{k} (n\varepsilon)^{-1},   & d\le k.
\end{cases}
\end{equation}
Finally, note that
\[
\sup_{X\in\Omega^n}\E\big[d_k(\mathrm{Unif}(X),\mathcal M(X))\big]
\;\ge\; \E_{\theta}\E\big[d_k(p_{X^\theta},\mathcal M(X^\theta))\big]
=\E_{\theta,Z}\big[d_k(p_\theta,p_Z )\big],
\]
so the lower bound \eqref{eq:data-risk-rate} for the Bayes risk with respect
to the distribution of $\theta$ also lower bounds the minimax risk. Taking an infimum over all $(\varepsilon, \delta)$-differentially private algorithm $\mathcal{M}$ finishes the proof.
\end{proof}

\section{Proof of Lemma~\ref{lem:isometry}}\label{sec:isometry}
\begin{proof}
Since $dx_i=-\sin\theta_i\,d\theta_i$ and $\sqrt{1-x_i^2}=\sin\theta_i$ for $\theta_i\in[0,\pi]$,
\[
  \int_{[-1,1]^d}|f(\bx)|^2\,d\mu_d(\bx)
  =\frac{1}{\pi^d}\int_{[0,\pi]^d}|f(\cos\btheta)|^2\,d\btheta
  =\|Uf\|_2^2.
\]
For $n\ge1$, $U(\overline T_n)(\btheta)=\sqrt2\cos(n\btheta)=\psi_n(\btheta)$ and $U(\overline T_0)=1=\psi_0$. The orthogonality and completeness follow from the computation of the one-dimensional case.
\end{proof}

\section{Proof of Theorem~\ref{thm:main-d}}\label{sec:appendix_chebyshev}

For a multi-index $\balpha=(\alpha_1,\dots,\alpha_d)$, denote
\[|\balpha|\coloneqq\sum_i\alpha_i,\quad  \balpha!\coloneqq \prod_i \alpha_i!\] and for $K=(k_1,\dots,k_d)$, denote
$K^{2\balpha}\coloneqq \prod_{i=1}^d k_i^{2\alpha_i}$, $\|K\|_2^2\coloneqq \sum_i k_i^2$.

We prove Theorem~\ref{thm:main-d} with several lemmas. The outline of the proof of Theorem~\ref{thm:main-d} is to show the following inequalities:
\[
    \sum_{K}\|K\|_2^{2k}\,c_K^2
    \;\leq\; \sum_{K} K^{2\balpha} c_K^2
    \;=\; \|\partial_\btheta^\balpha g\|_2^2
    \;\le\; C_{d,k}'\, \Big(\max_{|\bbeta|\le k}\|\partial_x^\bbeta f\|_{L^\infty([-1,1]^d)}\Big)^2.
\]

As introduced in Lemma~\ref{lem:isometry}, the Chebyshev polynomials corresponds to Fourier basis after a step of change-of-variable. Therefore we will first study the $k$-smooth functions in $L^2([0,\pi], \pi^{-d}\dd \btheta)$. The following lemma is a classical result in Fourier analysis, proven by integration by parts. (See \citep[Section 3, Eq~(1.4), (1.23)]{taylor1996partial} for more details.)

\begin{lemma}
[Parseval with derivatives in $\btheta$-space, \citep{taylor1996partial}]\label{lem:parseval-deriv-d}
With $g(\btheta)=f(\cos\btheta)$ and $g=\sum_K c_K\psi_K\in L^2([0,\pi]^d)$, for each $\balpha\in\mathbb{N}^d$, $|\balpha| \leq k$, $\partial^{\balpha}_gg$ is continuous and 
\[
  \|\partial_\btheta^\balpha g\|_2^2=\sum_{K} K^{2\balpha} c_K^2.
\]
\end{lemma}

\begin{lemma}[Chain rule with bounded trigonometric coefficients]\label{lem:chain}
Fix $k\in\mathbb{N}$. For every $\balpha\in\mathbb{N}^d$ with $|\balpha|\le k$, there exist trigonometric
polynomials $P_{\balpha,\bbeta}(\btheta)$, indexed by multi-indices $\bbeta$ with $0\le \bbeta\le \balpha$
(coordinate-wise), such that
\begin{equation}\label{eq:chain-expansion}
  \partial_\btheta^\balpha g(\btheta)
  =\sum_{0\le\bbeta\le\balpha} P_{\balpha,\bbeta}(\btheta)\,\partial_\bx^\bbeta f(\cos\btheta),
\end{equation}
each $P_{\balpha,\bbeta}$ has degree at most $|\balpha|$ in every $\theta_i$,
and
\begin{equation}\label{eq:P-bound}
  \sup_{\btheta\in[0,\pi]^d}|P_{\balpha,\bbeta}(\btheta)|
  \ \le |\balpha|!\qquad \forall \; 0\le \bbeta\le\balpha.
\end{equation}
\end{lemma}

To show the infinity norm bound in Lemma~\ref{lem:chain}, we introduce the Bernstein inequality in harmonic analysis, which states that the infinity norm of a trigonometric polynomial can be controlled by its derivative. 
\begin{lemma}[Bernstein inequality \citep{Katznelson_2004} Ex.~7.15(c), (7.25)]
\label{lem:bernstein}
Consider a trigonometric polynomial $Q(\vartheta)=\sum_{n=-m}^m a_n e^{in\vartheta}$ of degree $\le m$.
Then
\[
  \sup_{\vartheta\in\R}|Q'(\vartheta)|
  \ \le\ m \sup_{\vartheta\in\R}|Q(\vartheta)|.
\]
\end{lemma}

\begin{proof}(Proof of Lemma~\ref{lem:chain})
Induct on $|\balpha|$. For $\balpha=0$, we have $P_{0,0}\equiv 1$.

Suppose \eqref{eq:chain-expansion} holds for some $\balpha$. Next we fix $j\in\{1,\dots,d\}$ and consider $\balpha+\be_j$ where $\be_j$ is the vector with one at $j$-th entry and zero elsewhere. Applying $\partial_{\theta_j}$ we have
\[
\partial_{\theta_j}\partial_\btheta^\balpha g
=\sum_{0\le\bbeta\le\balpha} \big(\partial_{\theta_j}P_{\balpha,\bbeta}\big)\,\partial_x^\bbeta f(\cos\btheta)
-\sum_{0\le\bbeta\le\balpha} (\sin\theta_j)P_{\balpha,\bbeta}(\btheta)\,\partial_x^{\bbeta+\be_j} f(\cos\btheta),
\]
since $\partial_{\theta_j}=\!-\sin\theta_j\,\partial_{x_j}$.
Relabel $\bbeta\mapsto \bbeta-\be_j$ in the second sum (interpreting $P_{\balpha,\bbeta-\be_j}\equiv0$ if $\beta_j=0$) 
to obtain the representation for $\balpha+\be_j$ with
\[
  P_{\balpha+e_j,\bbeta}=\partial_{\theta_j}P_{\balpha,\bbeta} - (\sin\theta_j)P_{\balpha,\bbeta-\be_j} \qquad \forall \; 0\leq \bbeta\leq \balpha.
\]
From the recursion formula, $P_{\balpha,\bbeta}$'s are trigonometric polynomials of degree
$\le|\balpha|$ in each coordinate $\theta_j$.

For the uniform upper bound, define
\(
  M_t \coloneqq  \max\{\sup_\btheta|P_{\balpha,\bbeta}(\btheta)|: |\balpha|=t,\,0\le\bbeta\le\balpha\}.
\)
Using Lemma~\ref{lem:bernstein} in the $j$-th variable and $|\sin\theta_j|\le1$,
\[
  \sup_\btheta|P_{\balpha+\be_j,\bbeta}|
  \le \sup_\btheta|\partial_{\theta_j}P_{\balpha,\bbeta}| + \sup_\btheta|P_{\balpha,\bbeta-e_j}|
  \le |\balpha| M_{|\balpha|} + M_{|\balpha|} \le (|\balpha|{+}1)M_{|\balpha|}.
\]
This yields $M_{t+1}\le (t+1)\,M_t$, and $M_0=1$, hence
$M_t\le t!$. This proves \eqref{eq:P-bound}.
\end{proof}

\begin{proposition}[Uniform bound for $\|\partial_\btheta^\balpha g\|_{L^2([0,\pi]^d)}$]\label{prop:gk}
For any $f\in C^k([-1,1]^d)$, $g=f(\cos \btheta)$, and index $\balpha$ with $|\balpha|\le k$,
\[
  \|\partial_\btheta^\balpha g\|_{L^2([0,\pi]^d)}
  \le |\balpha|!\, \left(1 + \frac{k}{d}\right)^d\,
  \max_{|\bbeta|\le k}\|\partial_x^\bbeta f\|_{L^\infty([-1,1]^d)}.
\]
\end{proposition}

\begin{proof}
Using the expansion formula \eqref{eq:chain-expansion}, we have 
\begin{align*}
  \|\partial_\btheta^\balpha g\|_{L^2([0,\pi]^d)} 
  & \leq \|\partial_\btheta^\balpha g\|_{L^\infty([0,\pi]^d)}  \\
  & \le \sum_{0\le\bbeta\le\balpha} \sup_\btheta|P_{\balpha,\bbeta}(\btheta)|\cdot 
       \|\partial_x^\bbeta f(\cos\btheta)\|_{L^\infty([0,\pi]^d)} \\
  & \leq \sum_{0\le\bbeta\le\balpha} \sup_\btheta|P_{\balpha,\bbeta}(\btheta)|\cdot
       \|\partial_x^\bbeta f\|_{L^\infty([-1,1]^d)}.
\end{align*}
Note that here the last inequality is due to our $L^2$ measure being normalized. As the number of $\bbeta$ with $0\le\bbeta\le\balpha$ is $\prod_{i=1}^d(\alpha_i+1)\le ( 1 + \frac{k}{d})^d$ by Jensen's inequality.
Applying the uniform bound \eqref{eq:P-bound} concludes the proof.
\end{proof}

\begin{proof}(Proof of Theorem~\ref{thm:main-d})
We have \begin{align}
    \sum_{K}\|K\|_2^{2k}\,c_K^2
    & \;\leq\; \left(\sum_{|\balpha|=k}\frac{k!}{\balpha!}\right)\cdot \max_{|\balpha|=k} \sum_{K} K^{2\balpha} c_K^2
    \\
    & \;=\; \left(\sum_{|\balpha|=k}\frac{k!}{\balpha!}\right) \cdot \max_{|\balpha|=k}\|\partial_\btheta^\balpha g\|_2^2 \\
    &\;\le\; C_{d,k}'\, \Big(\max_{|\bbeta|\le k}\|\partial_x^\bbeta f\|_{L^\infty([-1,1]^d)}\Big)^2
    \label{eq: inqualities}.
\end{align}
Here, the first inequality above is from
\[
  \|K\|_2^{2k}=\Big(\sum_{i=1}^d a_i\Big)^k
  =\sum_{|\balpha|=k}\frac{k!}{\balpha!}\ba^\balpha
  \le \sum_{|\balpha|=k} \frac{k!}{\balpha !} \ba^\balpha = \sum_{|\balpha|=k} \frac{k!}{\balpha !} K^{2\balpha}.
\]
The second step in \eqref{eq: inqualities} is showed in Lemma~\ref{lem:parseval-deriv-d} and the last step is proved in Proposition~\ref{prop:gk} with 
\begin{align*}
   C_{d,k}' &=  \sum_{|\balpha| =k } \frac{k!}{\balpha !} \cdot (k!)^2 \left(1 + \frac{k}{d}\right)^{2d}  = d^k (k!)^2 \left(1 + \frac{k}{d}\right)^{2d},
\end{align*}
where we used the multinomial theorem $\sum_{|\balpha|=k} \frac{k!}{\balpha !} = d^k$.
\end{proof}

\section{Multivariate Jackson's Theorem for $k$-smooth functions}
\label{sec:Jackson}

To show the multivariate Jackson's Theorem, we first consider the measure $\pi^{-1}\dd x$ on $[0,\pi]$ and introduce a powerful tool in the construction: 
\begin{lemma}[Jackson kernel, \citep{jackson1912approximation}]
\label{lem: Jackson kernel}
    For integer $m, k\geq 1$, there exists a even kernel function $J_1(x):\R\to \R$ with the following properties:
    \begin{enumerate}
        \item (Trigonometric polynomial) $J_1(x) = \sum_{i=0}^m \widehat J_1(i)\cos(ix)$;
        \item (Positivity) $J_1(x)\geq 0$;
        \item (Normalization) $\frac{1}{\pi}\int_{0}^\pi J_1(x)\dd x =1$;
        \item (Bounded moments) $\frac{1}{\pi}\int_{0}^\pi |x|^\ell J_1(x)\dd x \leq \frac{C}{m^\ell}$ for $\ell \leq k$.
    \end{enumerate}
\end{lemma}

\begin{lemma}[Jackson's Theorem for 1-dimensional $k$-smooth functions, \citep{jackson1912approximation}]
\label{lem:Jackson_k=1}
Let $m,k\geq 1$ and a even periodic function $g$ with period $2\pi$ satisfies $\|g^{(k)}\|_{L^\infty([0,\pi])} \leq 1$. Then for every integer $m\geq k$,
\[\tilde g_m(x) \coloneqq  g*J_1 = \frac1\pi \int_{0}^\pi J_1(y)g(x-y)\dd y,\] satisfies
\[ \|g-\tilde g_m\|_{L^\infty([0,\pi])} \ \le\ \frac{C_1^k}{m^k}.\]
Moreover, if we write the Fourier expansion 
\[g(x) = \sum_{i=0}^\infty c_i \cos(i x), \qquad \tilde g_m(x)=\sum_{i=0}^m \tilde c_i \cos(i x)\]
then $|\tilde c_i|\leq |c_i|, \;\forall \; 0\leq i\leq m.$
\end{lemma}
\begin{proof}
    The approximation in $\|\cdot\|_{L^\infty([0,\pi])}$ norm is already proved in \citep{jackson1912approximation} with the properties in Lemma~\ref{lem: Jackson kernel}. Here, we simply complete the proof for the bound on the coefficients. For the Jackson's kernel $J_1$, using the property in Lemma~\ref{lem: Jackson kernel}, we have for $0\leq i\leq m$,
    \[|\tilde c_i| = |c_i| \cdot \left|\frac 1\pi \int_{0}^\pi J_1(x)\cdot \cos(ix) \dd x \right| \leq |c_i|\cdot \frac 1\pi \int_{0}^\pi J_1(x) \dd x = |c_i|.\]
\end{proof}

We are now ready to extend the result to general $ d$-dimensional space and prove Lemma~\ref{lem:Jackson}.

\begin{proof}(Proof of Lemma~\ref{lem:Jackson})
    For any even periodic function $G(\btheta) = G(\theta_1,\dots, \theta_d)$, define operator $\mathcal A_i$ such that $\mathcal A_i(G)$ is to apply Lemma~\ref{lem:Jackson_k=1} on the variable $\theta_i$:
    \[\mathcal A_i(G)(\theta_1,\dots, \theta_d) = \frac1\pi \int_{0}^\pi J_1(y) G(\theta_1, \dots ,\theta_{i-1}, \theta_i-y,\theta_{i+1},\dots \theta_d)\dd y.\]
    Then $\mathcal A_i(G)$ is a trigonometric polynomial of degree $m$ in $\theta_i$. Note that $\mathcal A_i$ is linear. Therefore, if $g$ is a trigonometric polynomial of degree $m$ in $\theta_j$, then $\mathcal A_i(g)$ is a trigonometric polynomial of degree $m$ in both $\theta_i, \theta_j$. Moreover, since $J_1$ is normalized, we have
    \[\|\mathcal A_i(G)\|_{L^\infty([0,\pi]^d)} \leq \|G\|_{L^\infty([0,\pi]^d)} \cdot \frac{1}{\pi}\int_0^\pi J_1(x)\dd x = \|G\|_{\infty}.\]

    First, we will consider $g(\btheta) = f(\cos \btheta)$ as an even periodic function in each coordinate $\theta_i$. We will run the approximation process in each coordinate. Define $\tilde g(\btheta) = \mathcal A_d(\mathcal A_{d-1}(\cdots \mathcal (A_1(g(\btheta))))$. By our observation above, $\tilde g$ is a trigonometric polynomial of degree $m$ in all $\theta_i$, $1\leq i\leq d$. Hence, we can write
    \[\tilde g = \sum_{K\in \{0,\dots, m\}^d} \tilde c_K \prod_{i=1}^d \cos(k_i \theta_i).\]
    And we aim to show the $\|\cdot \|_{L^\infty([0,\pi]^d)}$ approximation and the entry-wise bound for $\tilde c_K$.

    By Lemma~\ref{lem:Jackson_k=1}, for any even periodic function $G(\btheta)$, there is \
    \[\|\mathcal A_i(G) - G\|_{L^\infty([0,\pi]^d)} \leq \frac{C_1^k}{m^k} \cdot \|\partial_i^k G\|_{L^\infty([0,\pi]^d)}\]
    and 
    \[\|\partial_j^k\mathcal A_i(G)\|_{L^\infty([0,\pi]^d)} = \|\mathcal A_i(\partial_j^kG)\|_{L^\infty([0,\pi]^d)} \leq \|\partial_j^k G\|_{L^\infty([0,\pi]^d)}, \quad \text{when } i\neq j.\]
    Therefore, by a telescoping sum argument, we will get 
    \[\|Q-g\|_{L^\infty([0,\pi]^d)}\leq \sum_{i=1}^d \frac{C_1^k}{m^k} \cdot \|\partial_i^k g\|_{L^\infty([0,\pi]^d)} \leq \frac{dC_1^k}{m^k} \cdot \max_{|\balpha|=k}\|\partial^\balpha g\|_{L^\infty([0,\pi]^d)}.\]
    
	For the entry-wise bound on the Fourier coefficients, note that
    \[Q(\btheta) = \frac{1}{\pi^d}\int_{[0,\pi]^d}\prod_{i=1}^d J_1(y_i) G(\theta_1-y_1, \dots, \theta_d-y_d) \dd \by.\]
    And we can compute the coefficients
    \begin{align*}
        \tilde c_K &= \frac{2^{\mathrm{nnz}(K)}}{(2\pi)^d} \int_{[-\pi,\pi]^d} Q(\btheta)\prod_{i=1}^d \cos(k_i \theta_i) \dd \btheta\\ 
        & = \frac{2^{\mathrm{nnz}(K)}}{(2\pi)^d} \int_{[-\pi,\pi]^d} \frac{1}{\pi^d}\int_{[0,\pi]^d}\prod_{i=1}^d J_1(y_i) G(\theta_1-y_1, \dots, \theta_d-y_d) \dd \by\prod_{i=1}^d \cos(k_i \theta_i) \dd \btheta \\
        & = \frac{1}{\pi^d} \int_{[0,\pi]^d} \prod_{i=1}^d J_1(y_i) \; \frac{2^{\mathrm{nnz}(K)}}{(2\pi)^d}\int_{[-\pi,\pi]^d}G(\theta_1-y_1, \dots, \theta_d-y_d) \prod_{i=1}^d \cos(k_i \theta_i) \dd \btheta \; \dd \by \\
        & = \frac{1}{\pi^d} \int_{[0,\pi]^d} \prod_{i=1}^d J_1(y_i) \; \frac{2^{\mathrm{nnz}(K)}}{(2\pi)^d}\int_{[-\pi,\pi]^d}G(u_1, \dots, u_d) \prod_{i=1}^d \cos(k_i (u_i+y_i)) \dd \bu \;\dd \by \\ 
        & \overset{\text{(a)}}{=} \frac{1}{\pi^d} \int_{[0,\pi]^d} \prod_{i=1}^d J_1(y_i)\cos(k_i y_i)  \dd \by \cdot \frac{2^{\mathrm{nnz}(K)}}{(2\pi)^d}\int_{[-\pi,\pi]^d}G(u_1, \dots, u_d) \prod_{i=1}^d \cos(k_i u_i) \dd \bu \\
        & = \frac{1}{\pi^d} \int_{[0,\pi]^d} \prod_{i=1}^d J_1(y_i)\cos(k_i y_i)  \dd \by \cdot c_K \\
        & = c_K\prod_{i=1}^d \Big(\frac 1\pi \int_0^\pi J_1(y_i)\cos(k_i y_i)  \dd y_i\Big).
    \end{align*}
    Here in the step (a), we expand $\cos(k_i (u_i+y_i)) = \cos(k_i u_i)\cos(k_i y_i) - \sin(k_i u_i)\sin(k_i y_i)$, and we drop the sine term as $G$ is an even function and its integral against $\sin(k_i y_i)$ would vanish. And then $|\tilde c_K|\leq |c_K|$ follows since 
    \[\Big|\frac 1\pi \int_0^\pi J_1(y_i)\cos(k_i y_i)  \dd y_i\Big| \leq \frac 1\pi \int_0^\pi J_1(y_i) \dd y_i =1.\]

    Finally, we will complete the proof by proving the properties for $f(\bx)$. By changes of variables we have $f(\bx) = \sum_{K} c_K \overline T_K(\bx)$, and we define 
    \[\tilde f(\bx) = \sum_{K\in \{0,\dots ,m \}^d} \tilde c_K \overline T_K(\bx).\]
    Then $|\tilde c_K|\leq |c_K|$ still holds. Moreover, we have 
    \[\|f-\tilde f\|_{L^\infty([-1,1]^d)} = \|g-\tilde g\|_{L^\infty([0,\pi]^d)} \leq \frac{dC_1^k}{m^k} \cdot \max_{|\balpha|=k}\|\partial^\balpha g\|_{L^\infty([0,\pi]^d)}.
    \]
    In the proof of Proposition~\ref{prop:gk}, we have shown that 
    \[\|\partial^\balpha g\|_{L^\infty([0,\pi]^d)} \leq |\balpha|!\cdot \left(1 + \frac{k}{d}\right)^d\,
    \max_{|\bbeta|\le k}\|\partial_x^\bbeta f\|_{L^\infty([-1,1]^d)} \leq k!\cdot e^k.\]
    Taking the maximum over $\balpha$ concludes the proof. 
\end{proof}


\section{Proof of Theorem~\ref{thm: IPM_Chebyshev} }\label{sec:IPM_proof}
\begin{proof} For any function $f\in \mathcal F_k$, let
 $\tilde f=\sum_K\tilde c_K \overline T_K$ be the multivariate Chebyshev polynomial in  in Lemma~\ref{lem:Jackson}. We have
    \begin{align*}
        \left|\<f,p-q\>\right| &\leq \left|\<\tilde{f}, p-q\>\right| + \left|\<f- \tilde{f}, p-q\>\right|\\
        & \leq \left|  \<\sum_{K}\tilde{c}_K\overline{T}_K(x), p-q\>\right|  + \left|\<f- \tilde{f}, p\>\right| + \left|\<f- \tilde{f}, q\>\right|\\
        &\leq \sum_{K} |\tilde c_K| \cdot \left| \<\overline{T}_K(x), p-q\>\right| + \frac{d C_k^{\mathrm{Jac}}}{m^k} + \frac{dC_k^{\mathrm{Jac}}}{m^k}.
    \end{align*}
    For the first term, note that $\<\overline{T}_K(x), p-q\> =  \E_{X\sim p} \overline{T}_K(X) - \E_{X\sim q} \overline{T}_K(X)$ is the difference of the $K$-th Chebyshev moments of $p,q$. By our assumption as well as the Cauchy-Schwarz inequality, 
    \begin{align*}
         \sum_{K} |\tilde c_K| \cdot \left| \<\overline{T}_K(x), p-q\>\right| &\leq \sum_{K\in \Z_{\geq0}^d} |c_K| \cdot \left| \<\overline{T}_K(x), p-q\>\right|\\
         & = \sum_{K\neq 0} \|K\|_2^{k}|\tilde c_K| \cdot \frac{\left| \<\overline{T}_K(x), p-q\>\right|}{\|K\|_2^{k}} \\
         &\leq \left(\sum_{K\neq 0} \|K\|_2^{2k}|\tilde c_K|^2\right)^{1/2} \left(\sum_{K\neq 0}\frac{\left| \<\overline{T}_K(x), p-q\>\right|^2}{\|K\|_2^{2k}}\right)^{1/2}\\
         &\leq \sqrt{C_{d,k}'} \cdot \Gamma,
    \end{align*}
    where in the last equality we use Theorem~\ref{thm:main-d}, which provides the decay rate of the Chebyshev moments of any $k$-th differential function.
\end{proof}

\section{Estimation for $S$.}\label{sec:S} 

\begin{lemma}\label{lem: est_S}
Let \[S\coloneqq \sum_{K \in \{0,\dots, m\}^d \setminus \{\0\}} \frac{1}{\|K\|_2^k}.\] Then the following holds:
    \begin{equation*}
  S \;\le\;
  \begin{cases}
    d\,2^{d-1}\!\left(1+\dfrac{m^{\,d-k}-1}{d-k}\right), & \text{if } k<d,\\
    d\,2^{d-1}\,(1+\log m), & \text{if }k=d,\\
    d\,2^{d-1}\!\left(1+\dfrac{1}{k-d}\right), &\text{if }k>d.
  \end{cases}
\end{equation*}
\end{lemma}
\begin{proof}
For $j=1,\dots,m$, define the shell
\begin{equation*}
  \mathcal A_j := \bigl\{K\in\{0,\ldots,m\}^d\setminus \{\mathbf 0 \}:\ \|K\|_\infty=j\bigr\},
  \qquad \|K\|_\infty:=\max_{1\le i\le d}|K_i|.
\end{equation*}
Then
\begin{equation*}
  S = \sum_{j=1}^m \ \sum_{K\in\mathcal A_j}\frac{1}{\|K\|_2^{\,k}}.
\end{equation*}
For $K\in\mathcal A_j$ we have $j\le \|K\|_2 \le \sqrt d\,j$, hence
\begin{equation*}
  \frac{1}{\|K\|_2^{\,k}} \le \frac{1}{j^{k}} \qquad\Rightarrow\qquad
  S \le \sum_{j=1}^m \frac{|\mathcal A_j|}{j^{k}}.
\end{equation*}
Points with $\|K\|_\infty\le j$ form a $(j+1)^d$ grid, hence $|\mathcal A_j| = (j+1)^d - j^d$.
By the mean value theorem applied to $x\mapsto x^d$, there exists $\xi\in(j,j+1)$ such that
\begin{equation*}
  (j+1)^d - j^d = d\,\xi^{d-1} \le d\,(j+1)^{d-1}\le d\,2^{d-1}j^{d-1}\qquad \text{when }j\ge 1.
\end{equation*} Therefore
\begin{equation}\label{eq:stars}
  S \;\le\; d\,2^{\,d-1}\sum_{j=1}^m j^{\,d-1-k}.
\end{equation}
Let $\alpha:=d-1-k$. Using the integral bound:
\begin{itemize}
  \item If $k<d$ (i.e.\ $\alpha>-1$), then
  \begin{equation*}
    \sum_{j=1}^m j^{\alpha} \;\le\; 1 + \int_{1}^{m} x^{\alpha}\,dx
    \;=\; 1 + \frac{m^{\alpha+1}-1}{\alpha+1}
    \;=\; 1 + \frac{m^{\,d-k}-1}{d-k}.
  \end{equation*}
  Substituting into \eqref{eq:stars} gives
  \begin{equation*}
    S \;\le\; d\,2^{\,d-1}\!\left(1+\frac{m^{\,d-k}-1}{d-k}\right).
  \end{equation*}

  \item If $k=d$ (i.e.\ $\alpha=-1$), then
  \begin{equation*}
    \sum_{j=1}^m \frac{1}{j} \;\le\; 1 + \int_{1}^{m} \frac{dx}{x}
    \;=\; 1+\log m,
  \end{equation*}
  and \eqref{eq:stars} yields
    $S\le d2^{\,d-1}\,(1+\log m)$.

  \item If $k>d$ (i.e.\ $\alpha<-1$), then
  \begin{equation*}
    \sum_{j=1}^{\infty} j^{\alpha} \;\le\; 1 + \int_{1}^{\infty} x^{\alpha}\,dx
    \;=\; 1 + \frac{1}{k-d},
  \end{equation*}
  hence $\sum_{j=1}^{m} j^{\alpha} \le 1 + \frac{1}{k-d}$ for all $m\ge 1$ and
  \begin{equation*}
    S \le d2^{d-1}\!\left(1+\frac{1}{k-d}\right).
  \end{equation*}
\end{itemize}
This finishes the proof for all three cases.
\end{proof}

\section{Proof of Theorem~\ref{thm: convergence_rate} and the choice of $m'$ in Algorithm~\ref{alg: main}} 
\label{sec: convergence_rate}

\begin{proof}
    We will apply Theorem~\ref{thm: IPM_Chebyshev} to prove the theorem. Let the parameter $m$ in Theorem~\ref{thm: IPM_Chebyshev} be determined later. For any nonzero $K\in \mathbb N^d$, we have a uniform upper bound 
    \begin{align*}
        \E  \left\vert \E_{X\sim {q_n}} \overline{T}_K(X) - \E_{X\sim q} \overline{T}_K(X)\right\vert ^2 &= \E \left\vert \frac{1}{n} \sum_{i=1}^n\left[\overline{T}_K(\bx_i) - \E \, \overline{T}_K(\bx_i)\right]\right\vert ^2\\
        & = \frac{1}{n^2} \mathrm{Var}\left(\sum_{i=1}^n\overline{T}_K(\bx_i)\right)\\
        & = \frac{1}{n}\mathrm{Var}\left(\overline{T}_K(\bx_1)\right) \\
        &\leq \frac{1}{n} \E \,\overline{T}_K(\bx_1)^2 \leq \frac{2^{d}}{n}.
    \end{align*}
    Hence 
    \[(\E \Gamma)^2 \leq {\E \Gamma^2} \leq \frac{2^{d}}{n}\sum_{K\in \{0,...,m\}^d\setminus\{0\}} \frac{1}{\|K\|_2^{2k}} \leq \begin{cases}
        \frac{d\,2^{2d-1}}{n}\!\left(1+\dfrac{m^{\,d-2k}-1}{d-2k}\right), & \text{if } 2k<d,\\
        \frac{d\,2^{2d-1}}{n}\,(1+\log m), & \text{if }2k=d,\\
        \frac{d\,2^{2d-1}}{n}\!\left(1+\dfrac{1}{2k-d}\right), &\text{if }2k>d.
    \end{cases}\]
    Here we substitute our estimate for $S$ with parameter $2k$ in Lemma~\ref{lem: est_S}. As a result, applying the conclusion from Theorem~\ref{thm: IPM_Chebyshev}, we have 
    \[\E d_k(q_n ,q) \leq  \frac{2C_k^{\mathrm{Jac}}\cdot d}{m^k} + \sqrt{C_{d,k}'}\cdot \E\Gamma.\]\
    \begin{enumerate}
        \item When $2k<d$, we can take $m = n^{1/d}/C$ for some large constant $C>0$ and 
        \[\E d_k(q_n ,q) \leq  \frac{2C_k^{\mathrm{Jac}}\cdot d}{m^k} + \sqrt{C_{d,k}'}\cdot \sqrt{\frac{2d2^{2d-1}m^{d-2k}}{n}} \leq \frac{C_{d,k}''}{n^{k/d}}.\]
        \item When $2k=d$, we can take $m=n$ and 
         \[\E d_k(q_n ,q) \leq  \frac{2C_k^{\mathrm{Jac}}\cdot d}{n^k} + \sqrt{C_{d,k}'}\cdot \sqrt{\frac{2d2^{2d-1}\log n}{n}} \leq C_{d,k}'' \sqrt{\frac{\log n}{n}}.\]
        \item When $2k>d$, as the bound of $\E \Gamma$ is not related to $m$, we can take $m$ sufficiently large and 
        \[\E d_k(q_n ,q) \leq 2\sqrt{C_{d,k}'\E\Gamma^2}\cdot \leq \frac{C_{d,k}''}{\sqrt{n}}.\]
    \end{enumerate}
    Here we can take $C_{d,k}'' = (C k\sqrt{d})^k$ for some large constant $C$.
\end{proof}

\begin{remark}[The choice of $m'$ in Algorithm~\ref{alg: main}]\label{rmk:choiceofm}
    The best choice of $m'$ is to take a small integer such that $\E d_k(p_Y, q) \lesssim (\varepsilon n)^{-\min\{1, k/d\}}$ (an extra $\log(\varepsilon n)$ factor may apply when $d=k$) in the accuracy discussion of Theorem~\ref{thm:main}. Considering the regimes in Theorem~\ref{thm:main} and also Theorem~\ref{thm: convergence_rate}, a more detailed discussion of the choices of $m'$ can be made depending on the values of $d,k$ and $2k$, which gives five cases in total: $d<k, d=k, k<d<2k, d=2k$ and $d>2k$. We omit sharper computations for each case and use a simpler version as shown in \eqref{eq: m'choice}.
\end{remark}

\section{Computational complexity of Algorithm~\ref{alg: main}}\label{sec:complexity}
    The running time of  Algorithm~\ref{alg: main} is dominated by the optimization step, which solves a constrained least-squares problem involving $m^{kd}$ variables. We can write the optimization step as solving the quadratic programming problem 
    \begin{align*}
        & \min_{\bx\in \R^N} \|\bA\bx-\bb\|_2^2 \\
        & \;\;\mathrm{s.t.} \quad \mathbf{1}^\top \bx = 1; \quad \bx \geq 0,
    \end{align*}
    where $\bA\in \R^{M\times N}$ and $M=(m+1)^d-1, N=m^{kd}$. 
    
    Solving such a quadratic programming problem with error $\eta$ gives time complexity $O(\frac{MN}{\sqrt{\eta}})$ according to \cite{beck2009fast}, which is due to $O(1/\sqrt{\eta})$ many iterations, and each iteration has time complexity $O(MN)$. 
    With the choice of $m$ specified in Section~\ref{sec: Accuracy} and $\eta=(\varepsilon n)^{-2}$, this leads to an overall time complexity of \[O\left((\varepsilon n)^{\min\{k,d\}+2}\right).\] We also refer to \citep{boyd2004convex} for background on alternative methods (e.g., Interior Point Method) in convex optimization, which also provide polynomial time complexity. Finally, we note that the optimization problem exhibits additional tensor structure that could potentially be exploited to further reduce the computational cost; exploring such improvements, however, is beyond the scope of this work.

\end{document}

%% file: macros.tex
\newcommand{\ip}[2]{\left\langle #1,#2\right\rangle}
\newcommand{\norm}[1]{\left\lVert #1\right\rVert}

\newcommand{\R}{\mathbb{R}}
\newcommand{\Z}{\mathbb{Z}}
\newcommand{\E}{\mathbb{E}}

\newcommand{\0}{{\mathbf{0}}}
\newcommand{\bx}{{\mathbf{x}}}
\newcommand{\by}{{\mathbf{y}}}
\newcommand{\bu}{{\mathbf{u}}}
\newcommand{\ba}{{\mathbf{a}}}
\newcommand{\be}{{\mathbf{e}}}
\newcommand{\bg}{{\mathbf{g}}}
\newcommand{\bA}{{\mathbf{A}}}
\newcommand{\bb}{{\mathbf{b}}}

\newcommand{\btheta}{{\boldsymbol{\theta}}}
\newcommand{\balpha}{{\boldsymbol{\alpha}}}
\newcommand{\bbeta}{{\boldsymbol{\beta}}}

\newcommand{\dd}{\mathrm{d}}

\DeclareMathOperator{\nnz}{nnz}

\definecolor{magenta}{RGB}{255,0,255}

\newtheorem{construction}[theorem]{Construction}

\newcommand{\<}{\left\langle}
\renewcommand{\>}{\right\rangle}